\RequirePackage{fix-cm}
\documentclass[smallcondensed]{svjour3}     
\smartqed  

\usepackage{graphicx}
\usepackage{mathptmx}      

\usepackage{hyperref}
\usepackage{amsfonts}
\usepackage{amsmath,booktabs,ctable,threeparttable}
\usepackage{amssymb,amsfonts,boxedminipage}
\usepackage[caption=false]{subfig}
\usepackage{multirow}
\usepackage{color,xcolor}
\usepackage{enumerate}

\usepackage{amsthm}
\usepackage{bm}
\usepackage{lineno}

\newtheorem{theoremmy}{Theorem}
\newtheorem{lemmamy}[theoremmy]{Lemma}
\newtheorem{proptmy}[theoremmy]{Property}
\newtheorem{exper}[theoremmy]{Experiment}
\numberwithin{equation}{section}
\numberwithin{theoremmy}{section}

\newcommand{\bsmallmatrix}[1]{\begin{bmatrix}\begin{smallmatrix}#1\end{smallmatrix}\end{bmatrix}}

\begin{document}

\title{On choices of formulations of computing
the generalized singular value decomposition of a large matrix pair
\thanks{Supported by
the National Natural Science Foundation of China (No.11771249).}}


\titlerunning{Choices of formulations of computing GSVD}

\author{Jinzhi Huang \and  Zhongxiao Jia}


\institute{Jinzhi Huang \at
              Department of Mathematical Sciences, Tsinghua University, 100084 Beijing, China \\
              \email{huangjz15@mails.tsinghua.edu.cn}           
           \and
           Zhongxiao Jia \at
              Corresponding author.
              Department of Mathematical Sciences, Tsinghua University, 100084 Beijing, China \\
              \email{jiazx@tsinghua.edu.cn}
}


\maketitle

\begin{abstract}
For the computation of the generalized singular value
decomposition (GSVD) of a large matrix pair
$(A,B)$ of full column rank, the GSVD is commonly formulated
as two mathematically equivalent generalized eigenvalue problems,
so that a generalized eigensolver can be applied to one of them and the
desired GSVD components are then recovered from the
computed generalized eigenpairs.
Our concern in this paper is, in finite precision arithmetic,
which generalized eigenvalue formulation is numerically
preferable to compute the desired GSVD components more accurately.
We make a detailed perturbation analysis on the two formulations
and show how to make a suitable choice between them.
Numerical experiments illustrate the results obtained.

\keywords{Generalized singular value decomposition \and
generalized singular value \and generalized singular vector
\and generalized eigenpair \and eigensolver \and perturbation analysis \and condition number}

\subclass{ 65F15 \and 65F35 \and 15A12 \and 15A18 \and 15A42}

\end{abstract}

\section{Introduction}\label{section:1}

The generalized singular value decomposition (GSVD) of a matrix pair
$(A,B)$ was first introduced by van Loan \cite{van1976generalizing}
and then developed by Paige and Saunders \cite{paige1981towards}.
It has become a standard decomposition and an important
computational tool \cite{golub2012matrix}, and has been extensively
used in a wide range of contexts, e.g., solutions of discrete linear
ill-posed problems \cite{hansen1998rank},
weighted or generalized least squares problems \cite{bjorck1996numerical},
information retrieval \cite{howland2003structure},
linear discriminant analysis \cite{park2005relationship},
and many others \cite{betcke2008generalized,chu1987singular,golub2012matrix,kaagstrom1984generalized,vanhuffel}.

Let $A\in\mathbb{R}^{m\times n}$ ($m\geq n$)
and $B\in\mathbb{R}^{p\times n}$ ($p\geq n$)
be large and possibly sparse matrices of full column rank,
i.e., ${\rm rank}(A)={\rm rank}(B)=n$.
The GSVD of $(A,B)$ is
\begin{equation}\label{GSVD}
  \left\{\begin{aligned}
    A&=UCX^{-1},\\
    B&=VSX^{-1},
  \end{aligned}\right.
  \quad\mbox{with} \quad
  \left\{\begin{aligned}
    C&={\rm diag}\{\alpha_1,\dots,\alpha_n\},\\
    S&={\rm diag}\{\beta_1,\dots,\beta_n\},
  \end{aligned}\right.
\end{equation}
where $X=[x_1,\dots,x_n]$ is nonsingular,
$U=[u_1,\dots,u_n]$ and $V=[v_1,\dots,v_n]$ are orthonormal,
and the positive numbers $\alpha_i$ and $\beta_i$ satisfy
$\alpha_i^2+\beta_i^2=1$, $i=1,\dots,n $.
We call such $(\alpha_i,\beta_i,u_i,v_i,x_i)$ a GSVD component
of $(A,B)$ with the generalized singular value
$\sigma_i=\frac{\alpha_i}{\beta_i}$,
the left generalized singular vectors $u_i$ and $v_i$,
and the right generalized singular vector $x_i$, $i=1,\dots,n$.
Denote the generalized singular value matrix of $(A,B)$ by
\begin{equation}\label{Sigma}
  \Sigma=CS^{-1}={\rm diag}\{\sigma_1,\dots,\sigma_n\}.
\end{equation}
Throughout this paper, we also refer to a scalar pair
$(\alpha_i,\beta_i)$ as a generalized singular value of $(A,B)$.
Particularly, we will denote by $\sigma_{\max}(A,B)$ and
$\sigma_{\min}(A,B)$ the largest and smallest generalized
singular values of $(A,B)$, respectively.
Obviously, the generalized singular values of the pair $(B,A)$ are
$\frac{1}{\sigma_i},\ i=1,2,\ldots,n$,
the reciprocals of those of $(A,B)$, and their
generalized singular vectors are the same as those of $(A,B)$.

For a prescribed target $\tau$, assume that the generalized
singular values of $(A,B)$ are labeled by
\begin{equation}\label{lablegsvl}
  |\sigma_1-\tau|\leq|\sigma_2-\tau|\leq\dots\leq |\sigma_{\ell}-\tau|
  < |\sigma_{\ell+1}-\tau|\leq\cdots\leq|\sigma_n-\tau|.
\end{equation}
Specifically, if we are interested in the $\ell$ smallest generalized
singular values of $(A,B)$ and/or the associated left and right
generalized singular vectors, we assume $\tau=0$ in \eqref{lablegsvl},
so that the generalized singular values are labeled in increasing order;
if we are interested in the $\ell$ largest generalized singular
values of $(A,B)$ and/or the corresponding generalized singular
vectors, we assume $\tau=+\infty$ in \eqref{lablegsvl}, so that
the generalized singular values are labeled in decreasing order.
More generally, once $\tau$ is bigger than the
largest generalized singular value, the $\ell$ generalized singular values
closest to $\tau$ are the largest ones of $(A,B)$.
In these two cases, the $\ell$ GSVD components $(\alpha,\beta,u,v,x)$
are called the extreme (smallest or largest) GSVD components of $(A,B)$.
Otherwise they are called $\ell$ interior GSVD components of
$(A,B)$ if the
given $\tau$ is inside the spectrum of the generalized
singular values of $(A,B)$.
We will abbreviate any one of the desired GSVD components as
$(\sigma,u,v,x)$ or $(\alpha,\beta,u,v,x)$ with the subscripts dropped.

For a large and possibly sparse matrix pair $(A,B)$, one kind of
approach to compute the desired GSVD components works on the pair directly.
Zha \cite{zha1996} proposes a joint bidiagonalization method to
compute the extreme generalized singular values $\sigma$ and
the associated generalized singular vectors $u,v,x$, which is a
generalization of Lanczos bidiagonalization type methods
\cite{jia2003implicitly,jia2010} for computing a partial ordinary
SVD of $A$ when $B=I$.
A main bottleneck of this method is that a large-scale least
squares problem with the coefficient matrix
$\begin{bmatrix}\begin{smallmatrix}A\\B\end{smallmatrix}\end{bmatrix}$
must be solved at each step of the joint bidiagonalization.
Jia and Yang \cite{jiayang2018} has made a further analysis
on this method and its variant, and provided more theoretical
supports for its rationale.

For the computation of GSVD,
a natural approach is to apply a generalized eigensolver to
the mathematically equivalent generalized eigenvalue problem of the cross product
matrix pair $(A^TA,B^TB)$ to compute the corresponding eigenpairs
$(\sigma^2,x)$ and then recover the desired GSVD components from the
computed eigenpairs.
However, because of the squaring of the generalized singular values
of $(A,B)$, for $\sigma$ small, the eigenvalues $\sigma^2$ of
$(A^TA,B^TB)$ are much smaller.
As a consequence, the smallest generalized singular values may be
recovered much less accurately and even may have no accuracy \cite{jia2006}.
Therefore, we will not consider such a formulation in this paper.

Another kind of commonly used approach formulates the GSVD
as a generalized eigenvalue problem \cite{hochstenbach2009jacobi},
where the Jacobi-Davidson method \cite{hochstenbach2004}
for the ordinary SVD problem has been adapted to a
mathematically equivalent formulation of the GSVD so that
a suitable generalized eigensolver
\cite{parlett1998symmetric,saad2011numerical,stewart2001matrix} can be
used. The approach then recovers the desired GSVD components.
Concretely, the two formulations proposed
in \cite{hochstenbach2009jacobi} transform the GSVD into
the generalized eigenvalue problem
of the augmented definite matrix pair
\begin{equation}\label{widehatAB}
  (\widehat{A},\widehat{B}):=\left(
  \begin{bmatrix}&A\\A^T&\end{bmatrix},
  \begin{bmatrix}I&\\&B^TB\end{bmatrix}\right),
\end{equation}
or the augmented definite matrix pair
\begin{equation}\label{widetildeBA}
  (\widetilde{B},\widetilde{A}):=\left(
  \begin{bmatrix}&B\\B^T&\end{bmatrix},
  \begin{bmatrix}I&\\&A^TA\end{bmatrix}\right).
\end{equation}
We will give detailed relationships between the GSVD of $(A,B)$ and
the generalized eigenpairs of $(\widehat{A},\widehat{B})$ and
$(\widetilde{B},\widetilde{A})$ in the next section.
One then applies a generalized eigensolver to either of them,
computes the corresponding generalized eigenpairs, and recovers the desired
GSVD components from those computed generalized eigenpairs.

As will be clear next section, the nonzero eigenvalues of
$(\widehat A,\widehat B)$ and $(\widetilde B,\widetilde A)$
are $\pm\sigma_i$ and $\pm \frac{1}{\sigma_i}$, $i=1,2,\ldots,n$,
respectively.
Therefore, the largest or interior generalized singular values of
$(A,B)$ become the largest or interior eigenvalues of
$(\widehat A,\widehat B)$, and the smallest or interior generalized
singular values are the largest and interior eigenvalues
of $(\widetilde B,\widetilde A)$.
In principle, we may use a number of projection methods, e.g.,
Lanczos type methods, to compute the extreme GSVD components via
solving the generalized eigenvalue problem of
$(\widehat A,\widehat B)$ or $(\widetilde B,\widetilde A)$.
For a unified account of projection algorithms,
we refer to \cite{baiedit2000}.
For the computation of interior GSVD components of $(A,B)$,
we may employ the Jacobi-Davidson type method proposed in
\cite{hochstenbach2009jacobi}, referred as JDGSVD,
where at each step a linear system, i.e., the correction equation,
is solved iteratively and its approximate solution is
used to expand the current searching subspaces.
The JDGSVD method deals with the generalized eigenvalue
problem of \eqref{widehatAB}  or \eqref{widetildeBA},
computes some specific generalized eigenpairs,
and recovers the desired GSVD components from the converged
generalized eigenpairs.

As far as numerical computations are concerned, an important
question arises naturally: which of the {\em mathematically equivalent}
formulations \eqref{widehatAB} and \eqref{widetildeBA} is
{\em numerically preferable}, so that the desired GSVD components
can be computed more accurately?
In this paper, rather than propose or develop any
numerical algorithm
for computing the desired $\ell$ GSVD components,
we focus on this question carefully, give
a deterministic answer to it, and suggest a definitive choice.
We first make a sensitivity analysis on
the generalized eigenpairs of \eqref{widehatAB} and \eqref{widetildeBA}.
Based on the results to be obtained, we establish accuracy estimates
for the approximate generalized singular values and the left and
right generalized singular vectors that are recovered from the
approximate generalized eigenpairs obtained.
Then by comparing the accuracy of the approximate GSVD components
recovered from the approximate generalized eigenpairs of
\eqref{widehatAB} and \eqref{widetildeBA}, we
make a correct choice between these two formulations.

This paper is organized as follows.
In Section \ref{section:2} we make a sensitivity analysis on the
generalized eigenvalue problems of the structured matrix pairs
$(\widehat A,\widehat B)$ and $(\widetilde B,\widetilde A)$,
respectively, and give error bounds for the generalized singular
values $\sigma$ and the generalized eigenvectors of
$(\widehat A,\widehat B)$ and $(\widetilde B,\widetilde A)$.
In Section \ref{section:3} we carry out a sensitivity analysis on
the approximate generalized singular vectors that are recovered
from the approximate generalized eigenpairs of
$(\widehat A,\widehat B)$ and $(\widetilde B,\widetilde A)$.
Based on the results and analysis, we conclude that \eqref{widetildeBA}
is preferable to compute the GSVD more accurately
when $A$ is well conditioned and $B$ is ill conditioned, and
\eqref{widehatAB} is preferable when $A$ is ill conditioned and $B$
is well conditioned.
In Section \ref{section:4} we propose a few practical choice
strategies on \eqref{widehatAB} and \eqref{widetildeBA}.
In Section \ref{section:6} we report the numerical experiments.
We conclude the paper in Section~\ref{section:7}.

Throughout this paper, denote by $\|\cdot\|$ the 2-norm of
a vector or matrix and $\kappa(C)=\sigma_{\max}(C)/\sigma_{\min}(C)$
the condition number of a matrix $C$ with $\sigma_{\max}(C)$
and $\sigma_{\min}(C)$ being the largest and smallest singular values
of $C$, respectively, and by $C^T$ the transpose of $C$.
Denote by $I_k$ the identity matrix of order $k$, by $0_k$ and
$0_{k\times l}$ the zero matrices of order $k$ and $k\times l$, respectively.
The subscripts are omitted when there is no confusion. We also denote
by $\mathcal{R}(C)$ the column space or range of $C$.
For brevity of our analysis and results, without loss of generality,
we suppose that $\|A\|$ and $\|B\|$ are comparable in size
and, furthermore, $A$ and $B$ have already been scaled so
that their 2-norms are of $\mathcal{O}(1)$, that is,
$\|A\|\approx\|B\|\approx 1$ roughly, meaning
that $\sigma_{\min}^{-1}(A)=\|A^{\dagger}\|\approx\kappa(A)$ and
$\sigma_{\min}^{-1}(B)=\|B^{\dagger}\|\approx\kappa(B)$ roughly and
the conditioning of $A$ and $B$ is reflected by $\sigma_{\min}(A)$ and
$\sigma_{\min}(B)$, respectively.

\section{Perturbation analysis of generalized eigenvalue problems
and the accuracy of generalized singular values}\label{section:2}

The generalized eigendecompositions of the matrix
pairs $(\widehat A,\widehat B)$  and $(\widetilde B,\widetilde A)$
are closely related to the GSVD of $(A,B)$ in the following way,
which is straightforward to verify.

\begin{lemmamy}\label{lemma:1}
Let the GSVD of $(A,B)$ be defined by \eqref{GSVD} with the
generalized singular values defined by \eqref{Sigma}.
Let $U_{\perp}\in\mathbb{R}^{m\times (m-n)}$
and $V_{\perp}\in\mathbb{R}^{p\times (p-n)}$ be such that
$[U,U_{\perp}]$ and $[V,V_{\perp}]$ are orthogonal.
Then the matrix pairs $(\widehat A,\widehat B)$ and
$(\widetilde B,\widetilde A)$ defined by \eqref{widehatAB} and
\eqref{widetildeBA} have the generalized eigendecompositions
\begin{equation}\label{gen-eg}
  \widehat{A}Y=\widehat{B} Y \widehat{\Sigma}
  \quad\mbox{and}\quad
  \widetilde{B}Z=\widetilde{A}Z \widetilde{\Lambda},
\end{equation}
respectively, where
\begin{equation}\label{eigpair-AB}
  \widehat\Sigma=\begin{bmatrix}\begin{smallmatrix}
  \Sigma&&\\&-\Sigma&\\&&0\end{smallmatrix}\end{bmatrix},
  \quad   \quad
  Y=\begin{bmatrix}\frac{1}{\sqrt2}U &\frac{1}{\sqrt2}U &U_{\perp}\\
  \frac{1}{\sqrt2}W &-\frac{1}{\sqrt2}W &0 \end{bmatrix}
\end{equation}
with $W=XS^{-1}$, and
\begin{equation}\label{eigpair-BA}
  \widetilde\Lambda=\begin{bmatrix}\begin{smallmatrix}
  \Lambda&&\\&-\Lambda&\\&&0\end{smallmatrix}\end{bmatrix},
  \quad\quad
  Z=\begin{bmatrix}\frac{1}{\sqrt2}V&\frac{1}{\sqrt2}V&V_{\perp}\\
  \frac{1}{\sqrt2}W^{\prime}&-\frac{1}{\sqrt2}W^{\prime}&0\end{bmatrix}
\end{equation}
with $\Lambda=\Sigma^{-1}=SC^{-1}$ and $W^{\prime}=XC^{-1}$.
Moreover, the columns of the eigenvector matrices $Y$ and $Z$ are
$\widehat B$- and $\widetilde A$-orthonormal, respectively, i.e.,
\begin{equation}\label{semiorth}
  Y^T\widehat B Y=I_{m+n},\quad\quad Z^T\widetilde A Z=I_{p+n}.
\end{equation}
\end{lemmamy}

Lemma \ref{lemma:1} illustrates that the GSVD component
$(\alpha,\beta,u,v,x)$ of $(A,B)$ corresponds to the
generalized eigenpair
\begin{equation}\label{eig-pair}
  (\sigma,y):=\left(\frac{\alpha}{\beta},\frac{1}{\sqrt2}
  \begin{bmatrix}u\\x/\beta\end{bmatrix}\right)
\end{equation}
of the augmented matrix pair $(\widehat A,\widehat B)$ with
the eigenvector $y$ satisfying $y^T\widehat A y=\sigma$ and
$y^T\widehat B y=1$ and the generalized eigenpair
\begin{equation}\label{eig-pair-2}
(\frac{1}{\sigma},z):=\left(\frac{\beta}{\alpha},\frac{1}{\sqrt2}
  \begin{bmatrix}v\\x/\alpha\end{bmatrix}\right)
\end{equation}
of the augmented matrix pair $(\widetilde B,\widetilde A)$ with the
eigenvector $z$ satisfying $z^T\widetilde B z=\frac{1}{\sigma}$ and
$z^T\widetilde A z=1$.
Therefore, the GSVD of $(A,B)$ is mathematically equivalent to the generalized eigendecompositions
\eqref{widehatAB} and \eqref{widetildeBA}.
In order to obtain some GSVD components  $(\alpha,\beta,u,v,x)$,
one can compute the corresponding
generalized eigenpairs $(\sigma,y)$ of $(\widehat A,\widehat B)$
or $(\frac{1}{\sigma},z)$ of
$(\widetilde B,\widetilde A)$ by applying a generalized
eigensolver to \eqref{widehatAB} or \eqref{widetildeBA},
and then recovers the desired GSVD components.

However, in numerical computations, we can
obtain only {\em approximate} eigenpairs of \eqref{widehatAB}
and \eqref{widetildeBA}, and thus recover only {\em approximate}
GSVD components of $(A,B)$.
As a result,
when {\em numerically backward stable} eigensolvers solve
the generalized eigenvalue
problems of \eqref{widehatAB} and \eqref{widetildeBA} with the
computed eigenpairs whose residuals have about the same size,
a natural and central concern is: which of the computed eigenpairs of
\eqref{widehatAB} and \eqref{widetildeBA} will yield more accurate
approximations to the desired GSVD components of $(A,B)$, that is,
which of \eqref{widehatAB} and \eqref{widetildeBA} is numerically
preferable to compute the GSVD components more accurately?

To this end, we need to carefully estimate the accuracy of the
computed eigenpairs and that of the recovered GSVD components.
Given a backward stable generalized eigensolver applied to
\eqref{widehatAB} and \eqref{widetildeBA}, let $(\widehat\sigma,
\widehat y)$ and $(\frac{1}{\widetilde{\sigma}},\widetilde z)$
be the computed approximations to $(\sigma,y)$ and
$(\frac{1}{\sigma},z)$, respectively.
Then $(\widehat \sigma,\widehat y)$ and
$(\frac{1}{\widetilde{\sigma}},\widetilde z)$
are the exact eigenpairs of some perturbed matrix pairs
\begin{equation}\label{widehatABp}
  (\widehat{\bm{A}},\widehat{\bm{B}})
 =(\widehat A+\widehat E,\widehat B+\widehat F)
  \quad\mbox{and}\quad
  (\widetilde{\bm{B}},\widetilde{\bm{A}})
 =(\widetilde B+\widetilde F,\widetilde A+\widetilde E),
\end{equation}
respectively, where the perturbations satisfy
\begin{equation}\label{perturbations1}
  \|\widehat E\|\leq\|\widehat A\|\epsilon, \quad
  \|\widehat F\|\leq\|\widehat B\|\epsilon
  \quad\mbox{and}\quad
  \|\widetilde F\|\leq\|\widetilde B\|\epsilon, \quad
  \|\widetilde E\|\leq\|\widetilde A\|\epsilon
\end{equation}
for $\epsilon$ small.
In applications, we typically have $\epsilon=\mathcal{O}(\epsilon_{\rm mach})$ or
$\epsilon=\mathcal{O}(\epsilon_{\rm mach}^{1/2})$ with
$\epsilon_{\rm mach}$ being the machine precision
\cite{baiedit2000,golub2012matrix,parlett1998symmetric,saad2011numerical,stewart2001matrix}.
Here in \eqref{widehatABp}, to distinguish from the exact
augmented matrices defined in \eqref{widehatAB} and \eqref{widetildeBA},
we have used the bold letters to denote the perturbed matrices.
Notice that the assumption
$\|A\|\approx\|B\|\approx1$ made in Section \ref{section:1} means
$\|\widehat A\|=\|A\|$, $\|\widetilde{A}\|=\max\{1,\|A\|^2\}$ and
$\|\widehat B\|=\max\{1,\|B\|^2\},\ \|\widetilde B\|=\|B\|$.
Therefore, the perturbations in \eqref{perturbations1} satisfy
\begin{equation}\label{perturbations2}
\max\left\{ \|\widehat E\|,\|\widehat F\|,
\|\widetilde E\|,\|\widetilde F\|\right\}\leq
\max\{\|A\|^2,\|B\|^2,1\}\epsilon.
\end{equation}

In what follows, we will analyze how accurate the computed
eigenpairs $(\widehat\sigma,\widehat y)$ and $(\frac{1}{\widetilde{\sigma}},
\widetilde z)$ are for a given small $\epsilon$.

\subsection{The accuracy of generalized singular values}

Stewart and Sun in the monograph \cite{stewart1990matrix}
use a chordal metric to
measure the distance between the approximate and exact eigenvalues
of a regular matrix pair.
Let $\widehat \sigma$ and $\sigma$ be the eigenvalues of
$(\widehat{\bm{A}},\widehat{\bm{B}})$ and $(\widehat A,\widehat B)$.
Then the chordal distance between them is
\begin{equation}\label{chordal}
 \chi(\widehat\sigma,\sigma)
  =\frac{|\widehat\sigma-\sigma|}
  {\sqrt{1+\widehat\sigma^2}\sqrt{1+\sigma^2}}.
\end{equation}

We present the following results.

\begin{theoremmy}\label{theorem:0o}
Let $(\sigma,y)$ and $(\frac{1}{\sigma},z)$ be simple eigenpairs of
$(\widehat A,\widehat B)$ and $(\widetilde B,\widetilde A)$, respectively,
and their approximations $(\widehat\sigma,\widehat y)$ and
$(\frac{1}{\widetilde\sigma},\widetilde z)$ be the exact eigenpairs
of the perturbed matrix pairs
$(\widehat{\bm{A}},\widehat{\bm{B}})=
(\widehat A+\widehat E,\widehat B+\widehat F)$ and
$(\widetilde{\bm{B}},\widetilde{\bm{A}})=
(\widetilde B+\widetilde F,\widetilde A+\widetilde E)$, respectively, with
the perturbations satisfying \eqref{perturbations1}.
Assume that the approximate eigenvectors $\widehat y$ and
$\widetilde z$ are decomposed in the unnormalized form of
$\widehat y=y+s$ and $\widetilde z=z+t$
with $y^T\widehat Bs=0$ and $z^T\widetilde At=0$.
Then the following error bounds hold:
\begin{eqnarray}
 \chi(\widehat\sigma,\sigma)&\leq&
  \frac{\|y\|^2(1+\delta_1)}{\sqrt{(y^T\widehat Ay)^2+(y^T\widehat By)^2}}
  \sqrt{\|\widehat E\|^2+\|\widehat F\|^2},  \label{acceigvalue1}\\
 \chi(\widetilde \sigma,\sigma)&\leq&
   \frac{\|z\|^2(1+\delta_2)}{\sqrt{(z^T\widetilde Az)^2+(z^T\widetilde Bz)^2}}
   \sqrt{\|\widetilde E\|^2+\|\widetilde F\|^2}, \label{acceigvalue2}
\end{eqnarray}
where $\delta_1=\frac{\|s\|}{\|y\|}$ and $\delta_2=\frac{\|t\|}{\|z\|}$.
\end{theoremmy}
\begin{proof}
By the fact that $\widehat Ay=\sigma\widehat By$ and $\widehat{\bm{A}}\widehat y
= \widehat\sigma\widehat{\bm{B}}\widehat y$,
we have
$\sigma=\frac{\widehat y^T\widehat Ay}{\widehat y^T\widehat By}
=\frac{y^T\widehat A\widehat y}{y^T\widehat B\widehat y}$ and
$\widehat\sigma=\frac{y^T\widehat{\bm{A}}\widehat y}{ y^T\widehat{\bm{B}}
\widehat y}$.
Applying these two expressions to \eqref{chordal}, we obtain
\begin{equation}\label{Xsig1}
\chi(\widehat\sigma,\sigma)=
 \frac{|y^T\widehat{\bm{B}}\widehat y\cdot y^T\widehat A\widehat y-
 y^T\widehat{\bm{A}}\widehat y\cdot y^T\widehat B\widehat y|}
 {\sqrt{(y^T\widehat A\widehat y)^2+(y^T\widehat B\widehat y)^2}
 \sqrt{(y^T\widehat{\bm{A}}\widehat y)^2+
 (y^T\widehat{\bm{B}}\widehat y)^2}}.
\end{equation}
By $\widehat A=\widehat{\bm{A}}-\widehat E$ and
$\widehat B=\widehat{\bm{B}}-\widehat F$,
the nominator in the above equality satisfies
\begin{eqnarray*}
|y^T\widehat{\bm{B}}\widehat y\cdot y^T\widehat A \widehat y-
 y^T\widehat{\bm{A}}\widehat y\cdot y^T\widehat B \widehat y|
 &=&|y^T\widehat{\bm{B}}\widehat y\cdot y^T(\widehat{\bm{A}}-\widehat E) \widehat y-
 y^T\widehat{\bm{A}}\widehat y\cdot y^T(\widehat{\bm{B}}-\widehat F)\widehat y|\\
 &=&|y^T\widehat{\bm{B}}\widehat y\cdot y^T\widehat E \widehat y-
 y^T\widehat{\bm{A}}\widehat y\cdot y^T\widehat F\widehat y| \\
 &\leq&
\sqrt{(y^T\widehat{\bm{A}}\widehat y)^2+(y^T\widehat{\bm{B}}\widehat y)^2}
 \sqrt{(y^T\widehat E\widehat y)^2+(y^T\widehat F\widehat y)^2},
\end{eqnarray*}
applying which to \eqref{Xsig1} gives rise to
\begin{equation*}
\chi(\widehat\sigma,\sigma)\leq
 \frac{\sqrt{(y^T\widehat E\widehat y)^2+(y^T\widehat F\widehat y)^2}}
 {\sqrt{(y^T\widehat A\widehat y)^2+(y^T\widehat B\widehat y)^2}}\leq
 \frac{\|y\|\|\widehat y\|\sqrt{\|\widehat E\|^2+\|\widehat F\|^2}}
 {\sqrt{(y^T\widehat A\widehat y)^2+(y^T\widehat{B}\widehat y)^2}}.
\end{equation*}
Notice from $\widehat y=y+s$ with $s$ satisfying $y^T\widehat Bs=0$
and $y^T\widehat As=0$  that
$y^T\widehat A\widehat y=y^T\widehat A(y+s)=y^T\widehat Ay$ and
$y^T\widehat B\widehat y=y^T\widehat B(y+s)=y^T\widehat By$.
Moreover, it has $\|\widehat y\|\leq \|y\|+\|s\|=\|y\|(1+\delta_1)$
with $\delta_1=\frac{\|s\|}{\|y\|}$.
Applying these facts to the above inequality gives \eqref{acceigvalue1}.

Replacing $\widehat\sigma$, $\widehat y$, $y$, $(\widehat A,\widehat B)$ and
$(\widehat{\bm{A}},\widehat{\bm{B}})$
with $\frac{1}{\widetilde\sigma}$, $\widetilde z$, $z$,
$(\widetilde B,\widetilde A)$ and
$(\widetilde{\bm{B}},\widetilde{\bm{A}})$, respectively,
in \eqref{acceigvalue1}, and exploiting the invariance of the chordal
distance under
reciprocal, i.e.,
\begin{displaymath}
  \chi(\widetilde\sigma^{-1},\sigma^{-1})=
  \chi(\widetilde\sigma,\sigma),
\end{displaymath}
we obtain \eqref{acceigvalue2}.
\end{proof}

Obviously, it can be seen from the proof
that \eqref{acceigvalue1} and \eqref{acceigvalue2}
are independent of scalings of $y$, $\widehat{y}$ and $z$, $\widehat{z}$.
Therefore, our assumption in the theorem on the unnormalized decomposition form
of $\widehat{y}$ and $\widehat{z}$ is without loss of generality and
is only for brevity of the presentation.

For the scalars $\delta_1$ and $\delta_2$ in \eqref{acceigvalue1}
and \eqref{acceigvalue2}, we claim that
\begin{equation}\label{eta12}
\delta_1=\mathcal{O}(\epsilon)
\quad\mbox{and}\quad
\delta_2=\mathcal{O}(\epsilon)
\end{equation}
for a sufficiently small $\epsilon$ in \eqref{perturbations1}.
To show this precisely, without loss of generality,
we assume that the approximate eigenvectors
$y$ of $(\widehat A,\widehat B)$ and
$z$ of $(\widetilde B,\widetilde A)$ are scaled such that
$y^T\widehat By=1$ and $z^T\widetilde Az=1$.
Moreover, let the generalized eigenvalue and eigenvector matrices
of $(\widehat A,\widehat B)$ and $(\widetilde B,\widetilde A)$
defined by \eqref{gen-eg} be partitioned as
\begin{equation}\label{partition}
\widehat\Sigma=\bsmallmatrix{\sigma &\\&\widehat\Sigma_2},\qquad
Y=[y,Y_2],\qquad
\widetilde\Lambda=\bsmallmatrix{\frac{1}{\sigma}&\\&\widetilde\Lambda_2},\qquad
Z=[z,Z_2].
\end{equation}
Relation \eqref{semiorth} shows $Y_2^T\widehat By=0$, i.e.,
the columns of $Y_2$ form a basis of $(\widehat By)^{\perp}$,
and $s^T\widehat By=0$ indicates that we can write $s=Y_2h$ for
some $h\in\mathbb{R}^{m+n-1}$.
By \eqref{norm-h} to be proved later, we have
\begin{equation}\label{sy}
\|s\|\leq\|Y_2\|\|h\|\leq\frac{\|Y_2\|^2\|\widehat y\|}
{\min_{\mu_i\neq\sigma}|\mu_i-\widehat\sigma|}\|\widehat \sigma
\widehat F-\widehat E\|\leq\eta_1\|y\|(1+\delta_1)
\end{equation}
with $\eta_1=\frac{\|Y_2\|^2\|\widehat\sigma\widehat F-\widehat E\|}
{\min_{\mu_i\neq\sigma}|\mu_i-\widehat\sigma|}$
and $\mu_i$ being the eigenvalues of
$(\widehat A,\widehat B)$ other than $\sigma$.
If $\epsilon$ in \eqref{perturbations1} is sufficiently small
such that $\eta_1<1$, then from \eqref{sy} we obtain an explicit
bound for $\delta_1$:
\begin{displaymath}
 \delta_1=\frac{\|s\|}{\|y\|}
\leq\frac{\eta_1}{1-\eta_1}
=\mathcal{O}(\epsilon).
\end{displaymath}
In an analogous manner, we can obtain  $\delta_2=\mathcal{O}(\epsilon)$.

It is worthwhile to point out that some first order expansions are
derived for $\chi(\widehat\sigma,\sigma)$ for a general regular matrix pair
in \cite[p.291-4]{stewart1990matrix} but
the constants in the second order smaller terms are unknown.
The proofs of bounds \eqref{acceigvalue1} and \eqref{acceigvalue2}
have no special requirement on the matrix pairs and
thus are directly applicable to a general regular matrix pair by replacing
the transpose by the conjugate transpose and the scalars in
the denominators by their absolute values. In comparison with
those results in \cite[p.291-4]{stewart1990matrix}, however, our bounds
contain explicit second order smaller terms since
we have obtained the explicit bounds for $\delta_1$ and $\delta_2$.

Exploiting $y=\frac{1}{\sqrt2}\bsmallmatrix{u\\x/\beta}$ and
$z=\frac{1}{\sqrt2}\bsmallmatrix{u\\x/\alpha}$ in Theorem \ref{theorem:0o},
and keeping \eqref{eta12} in mind, we can present the following results.

\begin{theoremmy}\label{thm:1}
Let $(\sigma, y)$ and $(\frac{1}{\sigma},z)$ be the eigenpairs
of $(\widehat A,\widehat B)$ and $(\widetilde A,\widetilde B)$
corresponding to the GSVD component $(\alpha,\beta,u,v,x )$ of $(A,B)$.
Assume that their approximations $(\widehat \sigma,\widehat y)$
and $(\frac{1}{\widetilde{\sigma}},\widetilde z)$ are the generalized eigenpairs
of the perturbed $(\widehat{\bm{A}},\widehat{\bm{B}})$ and
$(\widetilde{\bm{B}},\widetilde{\bm{A}})$, respectively,
where the perturbations satisfy \eqref{perturbations2}.
If $\epsilon$ is sufficiently small,
the following error estimates
hold:
\begin{eqnarray}
 \chi(\widehat\sigma,\sigma)&\leq&
  \frac{(\|x\|^2+\beta^2)(1+\delta_1)}{2\beta}
  \sqrt{\|\widehat E\|^2+\|\widehat F\|^2},
  \label{accuray-sigma-1} \\
 \chi(\widetilde \sigma,\sigma)&\leq&
  \frac{(\|x\|^2+\alpha^2)(1+\delta_2)}{2\alpha}
   \sqrt{\|\widetilde E\|^2+\|\widetilde F\|^2},
  \label{accuray-sigma-2}
\end{eqnarray}
where $\delta_1=\mathcal{O}(\epsilon)$ and $\delta_2=\mathcal{O}(\epsilon)$.
\end{theoremmy}
\begin{proof}
It suffices to prove \eqref{accuray-sigma-1}, and the proof of
\eqref{accuray-sigma-2} is similar.
From Lemma~\ref{lemma:1}, notice that the eigenvector
$y$ of $(\widehat A,\widehat B)$
satisfies $y^T\widehat Ay=\sigma$ and $y^T\widehat B y=1$. From
$\sigma=\alpha/\beta$, $\alpha^2+\beta^2=1$ and $\|u\|=1$,
we have
\begin{displaymath}
  \frac{\|y\|^2}{\sqrt{(y^T\widehat Ay)^2+(y^T\widehat By)^2}}
  =\frac{1}{2}\frac{\|u\|^2+\frac{\|x\|^2}{\beta^2}}{\sqrt{1+\sigma^2}}
  =\frac{\|x\|^2+\beta^2}{2\beta}.
\end{displaymath}
Applying this and \eqref{eta12} to
\eqref{acceigvalue1} yields \eqref{accuray-sigma-1}.
\end{proof}

Notice from \eqref{perturbations2} that the perturbation terms
in the right hand sides of both \eqref{accuray-sigma-1} and
\eqref{accuray-sigma-2} are  no more than
the same $\mathcal{O}(\varepsilon)$.
Theorem \ref{thm:1} illustrates that the accuracy of the approximate
generalized singular value $\widehat\sigma$ and that of $\widetilde\sigma$
are determined by $\beta$ and $\|x\|$, and by $\alpha$ and $\|x\|$, respectively.
Apparently, a large $\|x\|$ could severely impair the accuracy
of both $\widehat \sigma$ and $\widetilde \sigma$.
Fortunately, the following bounds show that $\|x\|$ must be modest
under some mild conditions.

\begin{lemmamy}\label{prop:1}
Let $X$ be the right generalized singular vector matrix of $(A,B)$
as defined in \eqref{GSVD} and $x$ be an arbitrary column of $X$. Then
\begin{equation}\label{eq:X}
   \|X\|\leq\min\{\|A^{\dag}\|,\|B^{\dag}\|\}
   \quad\mbox{and}\quad
  \|X^{-1}\|\leq\sqrt{\|A\|^2+\|B\|^2},
\end{equation}
where the superscript $\dag$ denotes the Moore-Penrose generalized
inverse of a matrix, and
\begin{equation}\label{eq:x}
\frac{1}{\sqrt{\|A\|^2+\|B\|^2}}
\leq\|x\|\leq\min\{\|A^{\dag}\|,\|B^{\dag}\|\}.
\end{equation}
\end{lemmamy}

\begin{proof}
The bounds in \eqref{eq:X} and the upper bound for $\|x\|$ in
\eqref{eq:x} are from Theorem 2.3 of \cite{hansen1989regularization}.
Note that $x$ is a column of $X$.
Then the lower bound for $\|x||$ in \eqref{eq:x} follows from the fact that
\begin{equation*}
  \|x\|\geq\sigma_{\min}(X)=\|X^{-1}\|^{-1}. \hfill\qedhere
\end{equation*}
\end{proof}

Lemma~\ref{prop:1} indicates that, provided that one of $A$
and $B$ is well conditioned, $\|x\|$ must be modest.
In applications, to our best knowledge, there seems no case
that both $A$ and $B$ are simultaneously ill conditioned.
Therefore, without loss of generality, we will assume that at least
one of $A$ and $B$ is well conditioned. Then we have $\|x\|=\mathcal{O}(1)$.
Under this assumption, the stacked matrix
$\begin{bmatrix}\begin{smallmatrix}A\\B\end{smallmatrix}\end{bmatrix}$
must be well conditioned, too {\cite[Theorem~4.4]{stewart1990matrix}}.

Moreover, Theorem 2.4 of \cite{hansen1989regularization}
shows that provided $\begin{bmatrix}\begin{smallmatrix}A\\B\end{smallmatrix}\end{bmatrix}$
is well conditioned, the singular values of $A$ and those of $B$
behave like $\alpha_i$ and $\beta_i,i=1,2,\ldots,n$, correspondingly:
the ratios of the singular values of $A$ and $\alpha_i$
(resp. those of the singular values of $B$ and $\beta_i$), when
labeled by the same order, are bounded from below and above by
$\big\|\begin{bmatrix}\begin{smallmatrix}A\\B \end{smallmatrix}\end{bmatrix}^{\dagger}\big\|^{-1}$
and $\big\|\begin{bmatrix}\begin{smallmatrix}A\\B \end{smallmatrix}\end{bmatrix}\big\|$,
respectively. As a consequence, it is straightforward to
justify the following basic properties, which will play a
vital role in analyzing the results in this paper.

\begin{proptmy}\label{props}
Assume that at least one of $A$ and $B$ is well conditioned.

\begin{itemize}
\item If both $A$ and $B$ are well conditioned, no $\alpha_i$
    and $\beta_i$ are small. In this case, all the generalized singular
    values $\sigma_i$ of $(A,B)$ are neither large nor small.

\item If $A$ or $B$ is ill conditioned, there must be some
  small $\alpha_i$ or $\beta_i$, that is, some generalized singular
  values $\sigma_i$ must be small or large.
  Moreover, the small generalized singular values
  $\sigma_i=\alpha_i/\beta_i=\alpha_i(1-\alpha_i^2)^{-\frac{1}{2}}
  \approx\alpha_i$ for $A$ ill conditioned and the large
  $\sigma_i=(1-\beta_i^2)^{\frac{1}{2}}/\beta_i\approx 1/\beta_i$
  for $B$ ill conditioned.

\item If $A$ is ill conditioned and $B$ is well conditioned, all
    the $\sigma_i$ cannot be large but some of them are small;
    if $A$ is well conditioned and $B$ is ill conditioned, all the
    $\sigma_i$ cannot be small but some of them are large.
\end{itemize}
\end{proptmy}

Notice that $\alpha^2+\beta^2=1$ and$\alpha>0$, $\beta>0$. We have
\begin{displaymath}
\frac{1}{1+\frac{1}{\|x\|^2}}=
\frac{\|x\|^2}{\|x\|^2+1}<\frac{\|x\|^2+\beta^2}{\|x\|^2+\alpha^2}<
   \frac{\|x\|^2+1}{\|x\|^2}=1+\frac{1}{\|x\|^2}.
\end{displaymath}
Therefore, it follows from \eqref{eq:x}
that
\begin{equation}\label{fratio}
\frac{1}{1+\|A\|^2+\|B\|^2}
<\frac{\|x\|^2+\beta^2}{\|x\|^2+\alpha^2}<1+\|A\|^2+\|B\|^2.
\end{equation}
This, together with  the assumption $\|A\|\approx\|B\|\approx 1$,
shows that the lower and upper bounds are
roughly $\frac{1}{3}$ and 3, respectively,
and the ratio is thus very modest. When
at least one of $A$ and $B$ is well conditioned,
it is clear that the numerators $\|x\|^2+\beta^2$
and $\|x\|^2+\alpha^2$ in the constants in front of the perturbation terms
in bounds \eqref{accuray-sigma-1} and \eqref{accuray-sigma-2} are
not only modest but also very comparable in size. However, it is worthwhile to
remind that the lower and upper bounds in \eqref{fratio}
shows that the ratio $\frac{\|x\|^2+\beta^2}{\|x\|^2+\alpha^2}$
is always modest, independent
of the conditioning of $A$ and $B$. Furthermore,
relation \eqref{fratio} shows that it is the denominators
$2\beta$ and $2\alpha$ that
decide the size of the constants in front of the perturbation terms in bounds
\eqref{accuray-sigma-1} and \eqref{accuray-sigma-2}.
As a consequence, in terms of Theorem~\ref{thm:1}
and Property~\ref{props}, we can draw the following
conclusions for the accurate computation of $\sigma$:
\begin{itemize}
\item For $A$ and $B$ well conditioned, both \eqref{widehatAB} and
    \eqref{widetildeBA} work well.

\item If $A$ is well conditioned but $B$ is ill conditioned,
    \eqref{widetildeBA} is preferable to \eqref{widehatAB}.

\item If $A$ is ill conditioned but $B$ is well conditioned,
    \eqref{widehatAB} is better than \eqref{widetildeBA}.
\end{itemize}

\subsection{The accuracy of generalized eigenvectors}

In terms of the angles between the approximate and exact
eigenvectors, we present the following accuracy estimates for
the approximate eigenvectors of the symmetric definite matrix pairs in
\eqref{widehatAB} and \eqref{widetildeBA}.

\begin{theoremmy}\label{lemma:2}
With the notations of Theorem \ref{theorem:0o}, the following bounds hold:
\begin{eqnarray}
\sin\angle(\widehat y,y)&\leq& \frac{\|\widehat B^{-1}\|}
{\min_{\mu_i\neq\sigma}|\mu_i-\widehat \sigma|}
 \sqrt{\|\widehat E\|^2+\widehat\sigma^2\|\widehat F\|^2},\label{angle-y1}\\
\sin\angle(\widetilde z,z)&\leq& \frac{\|\widetilde A^{-1}\|}
{\widetilde\sigma\min_{\nu_i\neq\frac{1}{\sigma}}|\nu_i-\frac{1}{\widetilde\sigma}|}
 \sqrt{\|\widetilde E\|^2+\widetilde\sigma^2\|\widetilde F\|^2},\label{angle-z1}
\end{eqnarray}
where the $\mu_i$ are the eigenvalues of $(\widehat A,\widehat B)$
other than $\sigma$, and the $\nu_i$ are the eigenvalues of
$(\widetilde B,\widetilde A)$ other than $\frac{1}{\sigma}$.
\end{theoremmy}

\begin{proof}
By definition, we have
$(\widehat A+\widehat E)\widehat y=
\widehat\sigma(\widehat B+\widehat F)\widehat y$
with $\widehat y=y+s=y+Y_2h$ for some $h\in\mathbb{R}^{m+n-1}$
and the matrix $Y_2$ defined as in \eqref{partition}.
By a simple manipulation, we obtain
\begin{equation*}
  (\widehat A-\widehat\sigma\widehat B)(y+Y_2h)
  =(\widehat\sigma\widehat F-\widehat E)\widehat y.
\end{equation*}
Premultiplying $Y_2^T$ both hand sides of the above relation,
and noticing from \eqref{partition} and \eqref{semiorth}
that $Y_2^T\widehat Ay=0$, $Y_2^T\widehat By=0$ and
$Y_2^T\widehat AY_2=\widehat\Sigma_2$, $Y_2^T\widehat BY_2=I_{m+n-1}$, we obtain
\begin{equation*}
  (\widehat\Sigma_2-\widehat\sigma I)h=
  Y_2^T(\widehat\sigma\widehat F-\widehat E)\widehat y.
\end{equation*}
Taking $2$-norms on both hand sides in the above equality
and exploiting
$$
\|(\widehat\Sigma_2-\widehat\sigma I)h\|
\geq\min_{\mu_i\neq\sigma}|\mu_i-\widehat \sigma|\|h\|
$$
with $\mu_i$
being the eigenvalues of $(\widehat A,\widehat B)$ other
than $\sigma$ leads to
\begin{equation}\label{norm-h}
  \|h\|\leq\frac{\|Y_2\|\|\widehat y\|}{\min_{\mu_i\neq\sigma}
  |\mu_i-\widehat \sigma|}\|\widehat\sigma\widehat F-\widehat E\|.
\end{equation}

By definition, the sine of the angle between $\widehat y=y+s$ and $y$ satisfies
\begin{eqnarray} \label{angle-y-1}
 \sin\angle(\widehat y,y)=\frac{\|(I-\frac{yy^T}{y^Ty})
 (y+s)\|}{\|\widehat y\|}=\frac{\|(I-\frac{yy^T}{y^Ty})
 s\|}{\|\widehat y\|}
 \leq\frac{\|s\|}{\|\widehat y\|}.
\end{eqnarray}
Substituting $\|s\|=\|Y_2h\|\leq\|Y_2\|\|h\|$
and \eqref{norm-h} into \eqref{angle-y-1} yields
\begin{equation}\label{angle-y-2}
  \sin\angle(\widehat y,y)\leq\frac{\|Y_2\|^2}
  {\min_{\mu_i\neq\sigma}|\mu_i-\widehat\sigma|}
  \sqrt{\|\widehat E\|^2+\widehat\sigma^2\|\widehat F\|^2}.
\end{equation}
Notice that $\widehat B$ is positive definite and $Y_2$ satisfies
$Y_2^T\widehat BY_2=I_{m+n-k}$. We have
\begin{eqnarray*}
  \|Y_2\|^2&=&\|Y_2^TY_2\|=\|Y_2^T(\widehat B)^{\frac{1}{2}}
  (\widehat B)^{-1}(\widehat B)^{\frac{1}{2}}Y_2\|  \\
  &\leq&\|\widehat B^{-1}\|\|(\widehat B)^{\frac{1}{2}}
  Y_2\|^2=\|\widehat B^{-1}\|,
\end{eqnarray*}
applying which to \eqref{angle-y-2} gives \eqref{angle-y1}.

Following the same derivation, we obtain
\begin{equation*}
  \sin\angle(\widetilde z,z)\leq\frac{\|\widetilde A^{-1}\|}
{\min_{\nu_i\neq\frac{1}{\sigma}}|\nu_i-\frac{1}{\widetilde\sigma}|}
 \sqrt{\|\widetilde F\|^2+\frac{1}{\widetilde\sigma^2}\|\widetilde E\|^2}
\end{equation*}
with $\nu_i$ being eigenvalues of $(\widetilde B,\widetilde A)$ other than $\frac{1}{\sigma}$,
i.e., \eqref{angle-z1} holds.
\end{proof}

Theorem~\ref{lemma:2} gives accuracy estimates for the approximate
eigenvectors of the matrix pairs
$(\widehat A,\widehat B)$ and $(\widetilde B,\widetilde A)$.
It presents the results in the form of the structured matrix pairs
and their eigenvalues.
For our use in the GSVD context, substituting the definitions of
$(\widehat A,\widehat B)$ and $(\widetilde B,\widetilde A)$
in \eqref{widehatAB} and \eqref{widetildeBA} as well as their
eigenvectors in \eqref{eigpair-AB} and \eqref{eigpair-BA} into
Theorem \ref{lemma:2},
we can express the results more clearly in terms of
the generalized singular values
of $(A,B)$ and the matrices $A$ and $B$ themselves.

\begin{theoremmy}\label{theorem:2}
With the notations of Theorem \ref{thm:1}, the following results hold:
\begin{eqnarray}
 \sin\angle(\widehat y,y) &\leq&
 \frac{\max\{1,\|B^{\dag}\|^2\}}{\widehat\sigma\min_{\sigma_i\not=\sigma}
 \{|1-\frac{\sigma_i}{\widehat\sigma}|,1\}}
 \sqrt{\|\widehat E\|^2+\widehat\sigma^2\|\widehat F\|^2}, \label{angle-y-widehat-y}\\
 \sin\angle(\widetilde z,z) &\leq&
 \frac{\max\{1,\|A^{\dag}\|^2\}}{\ \ \min_{\sigma_i\not=\sigma}
 \{|1-\frac{\widetilde\sigma}{\sigma_i}|,1\}\ }
  \sqrt{\|\widetilde E\|^2+\widetilde\sigma^2\|\widetilde F\|^2}, \label{angle-z-widetilde-z}
\end{eqnarray}
where the $\sigma_i$ are the generalized singular values of $(A,B)$ other than $\sigma$.
\end{theoremmy}

\begin{proof}
Since the eigenvalues of $(\widehat A,\widehat B)$ are
$\pm \sigma_1,\pm\sigma_2,\dots,\pm\sigma_n$ and $m-n$ zeros,
we have
\begin{eqnarray}
  \min_{\mu_i\neq\sigma}|\mu_i-\widehat\sigma|
  &=&\min_{\sigma_i\neq\sigma}\{
  |\sigma_i-\widehat\sigma|,\widehat\sigma,\sigma+\widehat\sigma,\sigma_i+\widehat\sigma\}\nonumber\\
  &=&\min_{\sigma_i\neq\sigma}\{|\sigma_i-\widehat\sigma|,\widehat\sigma\}
  =\widehat\sigma\min_{\sigma_i\neq\sigma}
  \left\{\left|1-\tfrac{\sigma_i}{\widehat\sigma}\right|,1\right\}, \label{denominator}
\end{eqnarray}
where the $\sigma_i$ are the generalized singular values of $(A,B)$
other than $\sigma$.

On the other hand,
by definition \eqref{widehatAB} of $\widehat B$, we have
\begin{equation}\label{numerator1}
  \|\widehat B^{-1}\|=\sigma^{-1}_{\min}
  (\widehat B)=\frac{1}{
  \min\{1,\sigma^2_{\min}(B)\}}=\max\{1,\|B^{\dag}\|\}.
\end{equation}
Applying \eqref{denominator} and \eqref{numerator1} to
\eqref{angle-y1}, we obtain \eqref{angle-y-widehat-y}.

Notice that the eigenvalues of  $(\widetilde B,\widetilde A)$ are
$\pm\frac{1}{\sigma_1},\pm\frac{1}{\sigma_2},\dots,
\pm\frac{1}{\sigma_n}$ and $m-n$ zeros.
Following the same derivations as above, we obtain
\begin{eqnarray*}
  \sin\angle(\widetilde z,z) &\leq&\frac{\max\{1,\|A^{\dag}\|^2\}}{\widetilde\sigma\min_{\sigma_i\neq\sigma}
  \{|\frac{1}{\sigma_i}-\frac{1}{\widetilde\sigma}|,\frac{1}{\widetilde\sigma}\}}
  \sqrt{\|\widetilde E\|^2+\widetilde\sigma^2\|\widetilde F\|^2}\\
  &=&\frac{\max\{1,\|A^{\dag}\|^2\}}{\min_{\sigma_i\not=\sigma}
  \{|1-\frac{\widetilde\sigma}{\sigma_i}|,1\}}
  \sqrt{\|\widetilde E\|^2+\widetilde\sigma^2\|\widetilde F\|^2},
\end{eqnarray*}
which proves \eqref{angle-z-widetilde-z}.
\end{proof}

Denote $\widehat\sigma=\sigma(1+\omega_1)$ and
$\widetilde\sigma=\sigma(1+\omega_2)$ with
$\omega_1=\frac{\widehat\sigma-\sigma}{\sigma}$ and
$\omega_2=\frac{\widetilde\sigma-\sigma}{\sigma}$.
Assume that $\epsilon$ in \eqref{perturbations2} is sufficiently small.
Then from \eqref{chordal} and \eqref{accuray-sigma-1}--\eqref{accuray-sigma-2},
we have
$\omega_1=\mathcal{O}(\epsilon)$ and $\omega_2=\mathcal{O}(\epsilon)$.
For any generalized singular value $\sigma_i\neq\sigma$ of $(A,B)$,
it is straightforward to obtain
\begin{equation*}
1-\tfrac{\sigma_i}{\widehat\sigma}=(1-\tfrac{\sigma_i}{\sigma})
(1+\tfrac{\omega_1\sigma_i}{(1+\omega_1)(\sigma-\sigma_i)})
=(1-\tfrac{\sigma_i}{\sigma})(1+\mathcal{O}(\epsilon)),
\end{equation*}
and
$$
1-\tfrac{\widetilde\sigma}{\sigma_i}=
(1-\tfrac{\sigma}{\sigma_i})
(1-\tfrac{\omega_2\sigma}{\sigma_i-\sigma})=
(1-\tfrac{\sigma}{\sigma_i})(1+\mathcal{O}(\epsilon)).
$$
As a consequence, it holds that
\begin{eqnarray}
 \min_{\sigma_i\neq\sigma}\{|1-\tfrac{\sigma_i}{\widehat\sigma}|,1\}&=&
 \min_{\sigma_i\neq\sigma}\{|1-\tfrac{\sigma_i}{\sigma}|,1\}
 (1+\mathcal{O}(\epsilon)),\label{minima1}\\
 \min_{\sigma_i\neq\sigma}\{|1-\tfrac{\widetilde\sigma}{\sigma_i}|,1\}&=&
 \min_{\sigma_i\neq\sigma}\{|1-\tfrac{\sigma}{\sigma_i}|,1\}
 (1+\mathcal{O}(\epsilon)).\label{minima2}
\end{eqnarray}

For the minima in the right-hand sides of
\eqref{minima1} and \eqref{minima2},
we  have  the following result.

\begin{theoremmy}\label{prop:2}
Denote $\gamma_1=\min_{\sigma_i\neq\sigma}\{|1-\frac{\sigma_i}{\sigma}|,1\}$ and
$\gamma_2=\min_{\sigma_i\neq\sigma}\{|1-\frac{\sigma}{\sigma_i}|,1\}$
with $\sigma_i$ being the generalized singular values of $(A,B)$
other than $\sigma$. Then
\begin{equation}\label{ratio}
  \frac{1}{2} \leq\frac{\gamma_2}{\gamma_1}\leq 2.
\end{equation}
\end{theoremmy}

To prove this theorem, we need the following lemma.

\begin{lemmamy}\label{lemma:4}
Define $f(t)=\min\{|1-t|,1\}$ and $g(t)=\min\{|1-\frac{1}{t}|,1\}$
for $t\in(0,1)\cup(1,+\infty)$. Then
\begin{equation}\label{eq:13}
   \frac{1}{2}\leq\frac{g(t)}{f(t)}\leq 2.
\end{equation}
\end{lemmamy}

\begin{proof}
We classify nonnegative $t$ as three subintervals:
\begin{itemize}
  \item if $t\in(0,\frac{1}{2})$, then $f(t)=1-t$, $g(t)=1$
    and $\frac{g(t)}{f(t)}=\frac{1}{1-t}\in (1,2);$
  \item if $t\in[\frac{1}{2},1)\cup(1,2]$, then $f(t)=|1-t|$,
  $g(t)=|1-\frac{1}{t}|$
    and $\frac{g(t)}{f(t)}=\frac{1}{t}\in[\frac{1}{2},1)\cup (1,2];$
  \item if $t\in(2,+\infty)$, then $f(t)=1$, $g(t)=1-\frac{1}{t}$
    and $\frac{g(t)}{f(t)}=1-\frac{1}{t}\in(\frac{1}{2},1). $
\end{itemize}
Summarizing the above establishes \eqref{eq:13}.
\end{proof}

\begin{proof}[Proof of Theorem \ref{prop:2}]
Denote by $\sigma_{l}$ and $\sigma_{r}$ the generalized singular
values of $(A,B)$ that minimize $|1-\frac{\sigma_i}{\sigma}|$ and
$|1-\frac{\sigma}{\sigma_i}|$ over all the generalized singular
values $\sigma_i$ of $(A,B)$ other than $\sigma$, respectively.
Then $\gamma_1$ and $\gamma_2$ can be written as
\begin{eqnarray*}
\gamma_1&=&\min\{|1-\frac{\sigma_{l}}{\sigma}|,1\}
=f(\frac{\sigma_{l}}{\sigma}), \\
\gamma_2&=&\min\{|1-\frac{\sigma}{\sigma_{r}}|,1\}
=g(\frac{\sigma_{r}}{\sigma}),
\end{eqnarray*}
where the functions $f(\cdot)$ and $g(\cdot)$ are defined by
Lemma \ref{lemma:4}. Therefore, the ratio in \eqref{ratio} is
\begin{displaymath}
\frac{\gamma_2}{\gamma_1}=\frac{g(\frac{\sigma_{r}}{\sigma})}
{f(\frac{\sigma_{l}}{\sigma})}.
\end{displaymath}
By the definitions of $\sigma_{l}$ and $\sigma_{r}$, we have
\begin{equation}\label{eq:14}
  f(\frac{\sigma_{l}}{\sigma})\leq f(\frac{\sigma_{r}}{\sigma})
  \quad\mbox{and}\quad
  g(\frac{\sigma_{r}}{\sigma})\leq g(\frac{\sigma_{l}}{\sigma}).
\end{equation}
Combining \eqref{eq:14} with \eqref{eq:13}, we obtain
\begin{displaymath}
  \frac{1}{2}\leq\frac{g(\frac{\sigma_{r}}{\sigma})}
  {f(\frac{\sigma_{r}}{\sigma})}
  \leq\frac{g(\frac{\sigma_{r}}{\sigma})}
  {f(\frac{\sigma_{l}}{\sigma})}
  \leq\frac{g(\frac{\sigma_{l}}{\sigma})}
  {f(\frac{\sigma_{l}}{\sigma})}\leq 2,
\end{displaymath}
which completes the proof.
\end{proof}

Theorem~\ref{prop:2}, together with \eqref{minima1} and \eqref{minima2}, means
that the factors $\min_{\sigma_i\neq\sigma}\{|1-\frac{\sigma_i}{\widehat\sigma}|,1\}$ and
$\min_{\sigma_i\neq\sigma}\{|1-\frac{\widetilde\sigma}{\sigma_i}|,1\}$
in \eqref{angle-y-widehat-y} and \eqref{angle-z-widetilde-z} have
approximately the same size and both are approximately the relative
separation of the desired $\sigma$ from the other generalized singular values
of $(A,B)$. The bigger they are, i.e., the better the desired generalized
singular value $\sigma$ is separated from the others, the more
accurate the approximate eigenvectors of \eqref{widehatAB}
and \eqref{widetildeBA} are.

For a given $\epsilon$, \eqref{perturbations2} tells us that
$\sqrt{\|\widehat E\|^2+\widehat\sigma^2\|\widehat F\|^2}$ and
$\sqrt{\|\widetilde E\|^2+\widetilde\sigma^2\|\widetilde F\|^2}$
in \eqref{angle-y-widehat-y} and \eqref{angle-z-widetilde-z}
are approximately equal.
Therefore, Theorems~\ref{theorem:2}--\ref{prop:2} and
$\widehat\sigma=\sigma(1+\mathcal{O}(\epsilon))$,
$\widetilde\sigma=\sigma(1+\mathcal{O}(\epsilon))$
show that which of $\widehat y$ and $\widetilde z$ is more
accurate critically depends on the sizes
of $\frac{\max\{1,\|B^{\dag}\|^2\}}{\sigma}$
and $\max\{1,\|A^{\dag}\|^2\}$.
Keep in mind that $\|A\|\approx\|B\|\approx1$ means that
$\max\{1,\|A^{\dag}\|^2\}\approx\kappa^2(A)$ and
$\max\{1,\|B^{\dag}\|^2\}\approx\kappa^2(B)$.
Combining these results with Property~\ref{props}, for a
proper choice of \eqref{widehatAB} and \eqref{widetildeBA} for
computing eigenvectors more accurately,
we can draw the following conclusions with the arguments included.
\begin{itemize}
\item If $A$ and $B$ have roughly the same conditioning and
    both are well conditioned, then $\sigma$ cannot be large or
    small. In this case, both \eqref{widehatAB} and
    \eqref{widetildeBA} are proper formulations of computing the
    generalized eigenvectors $y$ and $z$ with similar accuracy.

\item For $B$ ill conditioned and $A$ well conditioned, assuming that
the $\beta_i$ are labeled in decreasing order, from
    Property~\ref{props}, since the pair $(A,B)$ has large
    generalized singular values $\sigma_i\approx 1/\beta_i$ but
    has no small one, it is known that
    $\|B^{\dag}\|\approx\frac{1}{\min_{i}\beta_i}\approx\max_{i}\sigma_i
    =\sigma_{\max}(A,B)$.
    Therefore, we have
    \begin{displaymath}
    \frac{\max\{1,\|B^{\dag}\|^2\}}{\sigma}\geq \frac{\|B^{\dag}\|^2}
    {\sigma_{\max}(A,B)}
    \sim \sigma_{\max}(A,B)\gg \max\{1,\|A^{\dag}\|^2\}
    \end{displaymath}
    for any $\sigma$. Therefore, \eqref{widetildeBA} is preferable to compute
    any eigenvector $z$ more accurately.

\item For $B$ well conditioned and $A$ ill conditioned,
    from Property~\ref{props}, since some generalized
    singular values $\sigma$ of $(A,B)$ are small but none
    is large, it is known that
    $\|A^{\dag}\|\approx \frac{1}{\min_{i}\alpha_i}
    \approx\frac{1}{\min_{i}\sigma_i}
    =\frac{1}{\sigma_{\min}(A,B)}$.
    Therefore, we always have
    \begin{displaymath}
      \frac{\max\{1,\|B^{\dag}\|^2\}}{\sigma}\leq
      \frac{\max\{1,\|B^{\dag}\|^2\}}{\sigma_{\min}(A,B)}
      \sim \frac{1}{\sigma_{\min}(A,B)}\ll \frac{1}{\sigma_{\min}^2(A,B)}
      \approx\max\{1,\|A^{\dag}\|^2\}
    \end{displaymath}
    for any $\sigma$. This means that \eqref{widehatAB} is
    preferable to compute any eigenvector $y$ more accurately.
\end{itemize}

Finally, we notice from Theorem~\ref{theorem:2} that $\hat y$ or
$\tilde z$ may have no accuracy at all whenever $\kappa(B)$ or $\kappa(A)$ is
as large as $\mathcal{O}(\epsilon_{\rm mach}^{-1/2})$, even though a backward stable
generalized eigensolver is applied to \eqref{widehatAB} or
\eqref{widetildeBA} and backward errors are $\mathcal{O}(\epsilon_{\rm mach})$.
For a large matrix pair $(A,B)$, iterative projection methods
are used to compute some specific GSVD components and stopping criteria
are typically $\mathcal{O}(\epsilon_{\rm mach}^{1/2})$, so that
backward errors are $\mathcal{O}(\epsilon_{\rm mach}^{1/2})$.
In this case, $\hat y$ or
$\tilde z$ may have no accuracy provided that $\kappa(B)$ or $\kappa(A)$ is as large
as $\mathcal{O}(\epsilon_{\rm mach}^{-1/4})$.

\section{The accuracy of generalized singular vectors}\label{section:3}

After applying a generalized eigensolver to the matrix
pair $(\widehat A,\widehat B)$ or $(\widetilde B,\widetilde A)$,
the computed eigenvalue $\widehat \sigma$ or $\widetilde\sigma$
provides an approximation to the desired generalized singular value
$\sigma$ directly.
However, the situation is complicated and more involved
for generalized singular vectors since the generalized eigenvector
\begin{displaymath}\label{yz}
  y=\frac{1}{\sqrt2}\begin{bmatrix}
   u\\x_{\beta}\end{bmatrix}
  \quad \mbox{or}\quad
  z=\frac{1}{\sqrt2}\begin{bmatrix}
   v\\x_{\alpha}\end{bmatrix}
\end{displaymath}
defined by \eqref{eig-pair} or \eqref{eig-pair-2}
is a stack of the normalized left generalized singular vector $u$ or
$v$ and the scaled right generalized singular vector
\begin{equation}\label{xalbe}
x_{\beta}=\frac{x}{\beta}  \quad \mbox{or}\quad x_{\alpha}=\frac{x}{\alpha}.
\end{equation}
We must recover approximations to the generalized singular vectors
$u,v,x$ from a computed approximate eigenvector $\widehat y$
or $\widetilde z$. For the GSVD components of $(A,B)$, our next task
is to determine which of $\widehat y$ and $\widetilde z$ delivers
more accurate approximations to $u,v$ and $x$ when the perturbations
$\widehat{E}, \widehat{F}$ and $\widetilde{E},\widetilde{F}$
in \eqref{widehatABp} approximately have the same size in norm.

For \eqref{widehatAB}, after a generalized eigensolver is run,
we write the converged approximate eigenvector as
$\widehat y=\frac{1}{\sqrt2}[\widehat u^T,\widehat x^T]^T$
with $\widehat u\in \mathbb{R}^{m}$ normalized to have unit length
and $\widehat x\in \mathbb{R}^{n}$.
Then $\widehat u$ and $\widehat \beta\widehat x$
provide approximations to the left generalized singular vector $u$
and the right generalized singular vector $x$, respectively,
with the computed $\widehat\sigma=\frac{\widehat\alpha}{\widehat\beta}$.
As for the left generalized singular vector $v$, since $Bx=\beta v$,
it is natural to take the unit length
$\widehat v=\frac{B\widehat x}{\|B\widehat x\|}$ as its approximation.

Analogously, for \eqref{widetildeBA}, we partition
$\widetilde z=\frac{1}{\sqrt2}[\widetilde v^T, \widetilde x^T]^T$
such that $\widetilde v\in\mathbb{R}^{p}$ is normalized to have unit length, $\widetilde x\in\mathbb{R}^{n}$,
and that $\widetilde v$ and  $\widetilde\alpha \widetilde x$
are approximations to the left generalized singular vector
$v$ and the right generalized singular vector $x$,
respectively, where the computed
$\frac{1}{\widetilde\sigma}=\frac{\widetilde\beta}{\widetilde\alpha}$.
Since $Ax=\alpha u$, we take the unit length
$\widetilde u=\frac{A\widetilde x}{\|A\widetilde x\|}$
as an approximation to $u$.

Previously we have derived error estimates on
$\sin\angle(\widehat y,y)$ and $\sin\angle(\widehat z,z)$ for
the approximate eigenvectors $\widehat y$ and $\widetilde z$.
Next we exploit them to estimate the accuracy of the
approximate generalized singular vectors
$(\widehat{u},\widehat{v},\widehat{\beta}\widehat{x})$ and $(\widetilde{u},\widetilde{v},\widetilde{\alpha}\widetilde{x})$ recovered
in the manner described above.
To this end, we prove the following lemma, which is a
generalization of Theorem 2.3 in \cite{jia2003implicitly}.

\begin{lemmamy} \label{lemma:3}
Assume that $a$ and $b$ are arbitrary nonzero vectors, and let
$a^{\prime}$ and $ b^{\prime}$ be approximations to them,
respectively. Then
\begin{equation} \label{angle-acbd}
  \|a\|^2\sin^2\angle(a^{\prime},a)+\|b\|^2\sin^2\angle(b^{\prime},b)
  \leq(\|a\|^2+\|b\|^2)\sin^2\angle\left(
  \begin{bmatrix}\begin{smallmatrix}a^{\prime}\\b^{\prime}
  \end{smallmatrix}\end{bmatrix},
  \begin{bmatrix}\begin{smallmatrix}a\\b
  \end{smallmatrix}\end{bmatrix}\right).
\end{equation}
Moreover, it holds that
\begin{eqnarray}
 \min\{\sin\angle(a^{\prime},a),\sin\angle(b^{\prime},b)\}
  &\leq&\sin\angle\left(
  \begin{bmatrix}\begin{smallmatrix}a^{\prime}\\b^{\prime}
  \end{smallmatrix}\end{bmatrix},
  \begin{bmatrix}\begin{smallmatrix}a\\ b
  \end{smallmatrix}\end{bmatrix}\right),   \label{bd1}\\
\sqrt{\sin^2\angle(a^{\prime},a)+\sin^2\angle(b^{\prime},b)}
 &\leq&\varrho\sin\angle\left(
  \begin{bmatrix}\begin{smallmatrix}a^{\prime}\\b^{\prime}
  \end{smallmatrix}\end{bmatrix},
  \begin{bmatrix}\begin{smallmatrix}a\\ b
  \end{smallmatrix}\end{bmatrix}\right),   \label{bd2}
\end{eqnarray}
where $\varrho=\sqrt{1+\max\left\{\tfrac{\|a\|^2}{\|b\|^2},
 \tfrac{\|b\|^2}{\|a\|^2}\right\}}.$
\end{lemmamy}
\begin{proof}
By definition, the sine of the angle between two vectors $a$
and $a^{\prime}$ satisfies
\begin{displaymath}
\|a\|\sin\angle(a^{\prime},a)=\min_{\mu}\|a-\mu a^{\prime}\|.
\end{displaymath}
A similar relation holds with $a$ and $a^{\prime}$ replaced by $b$
and $b^{\prime}$, respectively.
Combining these two relations with the inequality
\begin{displaymath}
\min_{\mu}\|a-\mu a^{\prime}\|^2+
\min_{\mu}\|b-\mu b^{\prime}\|^2\leq\min_{\mu}\left\|
\begin{bmatrix}\begin{smallmatrix}a\\b
\end{smallmatrix}\end{bmatrix}-\mu
\begin{bmatrix}\begin{smallmatrix}a^{\prime}\\b^{\prime}
\end{smallmatrix}\end{bmatrix}\right\|^2
\end{displaymath}
proves \eqref{angle-acbd}.
From \eqref{angle-acbd}, taking the smaller one of
$\sin\angle(a^{\prime},a)$ and $\sin\angle(b^{\prime},b)$
yields \eqref{bd1}.
It is also straightforward to obtain
\begin{eqnarray*}
 &\sin^2\angle(a^{\prime},a)+\sin^2\angle(b^{\prime},b)
  \leq (1+\tfrac{\|a\|^2}{\|b\|^2})\sin^2\angle\left(
  \begin{bmatrix}\begin{smallmatrix}a^{\prime}\\b^{\prime}
  \end{smallmatrix}\end{bmatrix},
  \begin{bmatrix}\begin{smallmatrix}a\\b
  \end{smallmatrix}\end{bmatrix}\right),
   \quad \mbox{if}\quad  \|a\|\geq\|b\|,  \\
 &\sin^2\angle(a^{\prime},a)+\sin^2\angle(b^{\prime},b)
  \leq (1+\tfrac{\|b\|^2}{\|a\|^2})\sin^2\angle\left(
  \begin{bmatrix}\begin{smallmatrix}a^{\prime}\\b^{\prime}
  \end{smallmatrix}\end{bmatrix},
  \begin{bmatrix}\begin{smallmatrix}a\\b
  \end{smallmatrix}\end{bmatrix}\right),
   \quad\mbox{if}\quad  \|a\|\leq\|b\|.
\end{eqnarray*}
Combining the above two inequalities gives rise to \eqref{bd2}.
\end{proof}

Taking $a=u$, $b=x_{\beta}$ and $a^{\prime}=\widehat u$,
$b^{\prime}=\widehat x$, bound \eqref{bd1} illustrates that at
least one of the recovered approximate generalized singular vectors
$\widehat u$ and $\widehat x$ is as accurate as $\widehat y$.
Since $\|u\|=1$, bound \eqref{bd2} indicates that if
$\|x_{\beta}\|=\mathcal{O}(1)$ then both $\widehat u$ and
$\widehat x$ have the same accuracy as $\widehat y$.
But bound \eqref{bd2} also states that if
$\|x_{\beta}\|$ is very small or large relative to $\|u\|=1$ then
one of $\widehat u$ and $\widehat x$ may have considerably poorer
accuracy than $\widehat y$ due to the large factor $\varrho$.
Fortunately, a very small $\|x_{\beta}\|$  is unlikely to 
happen as $\|x\|$ is always modest
under the assumption that at least one of $A$ and $B$ is well conditioned,
implying that $\|x_{\beta}\|=\frac{\|x\|}{\beta}$ cannot
be small as $0<\beta<1$. On the other hand,
when the largest GSVD components of $(A,B)$ are
required, a large $\|x_{\beta}\|$ definitely appears if $B$ is ill
conditioned since $\beta$ behaves like the singular values of $B$
and is small, as Property~\ref{props} shows.

Precisely, based on Lemma~\ref{lemma:3}, we can derive quantitative
accuracy estimates for the recovered approximate generalized
singular vectors, as will be shown below.

\begin{theoremmy}\label{accuracy1}
The scaled right generalized singular vector $x_{\beta}$, defined in
\eqref{xalbe}, of $(A,B)$ satisfies%
\begin{equation}
\|x_{\beta}\|\geq\frac{1}{\|B\|}. \label{xbeta}
\end{equation}
For the approximate generalized singular vectors
$\widehat u,\widehat v$ and $\widehat x$ recovered
from the approximate eigenvector
$\widehat y$ of \eqref{widehatAB}, it holds that
\begin{eqnarray}
   \sin\angle(\widehat x,x)&\leq&\sqrt{1+\frac{1}{\|x_{\beta}\|^2}}
    \sin\angle\left(\widehat y,y\right),       \label{acc-x-hat} \\
   \sin\angle(\widehat u,u)&\leq&\sqrt{1+\|x_{\beta}\|^2}
    \sin\angle\left(\widehat y,y\right),       \label{acc-u-hat} \\
   \sin\angle(\widehat v,v)&\leq&\|B\|\sqrt{1+\|x_{\beta}\|^2}
    \sin\angle\left(\widehat y,y\right).       \label{acc-v-hat}
\end{eqnarray}
\end{theoremmy}

\begin{proof}
From $Bx=\beta v$ and $\|v\|=1$ we have
\begin{displaymath}
\|x_{\beta}\|=\frac{\|x\|}{\beta}=
\frac{\|x\|}{\|Bx\|}\geq\frac{1}{\|B\|},
\end{displaymath}
which shows \eqref{xbeta}.

Take $a=u$, $b=x_{\beta}$ in Lemma \ref{lemma:3}.
Neglecting the first term in the left hand side of \eqref{angle-acbd},
we obtain
\begin{eqnarray*}
   \sin\angle(\widehat x, x)&=&\sin\angle(\widehat x, x_{\beta}) \leq
    \sqrt{1+\frac{\|u\|^2}{\|x_{\beta}\|^2}}
    \sin\angle\left(\widehat y,y\right) \\
    &=&\sqrt{1+\frac{1}{\|x_{\beta}\|^2}}
    \sin\angle\left(\widehat y,y\right),
\end{eqnarray*}
which proves \eqref{acc-x-hat}.

Neglecting the second term in the left hand side
of \eqref{angle-acbd} gives \eqref{acc-u-hat}.

As for $\widehat v=\frac{B\widehat x}{\|B\widehat x\|}$,
exploiting $Bx=\beta v$ with $\|v\|=1$ and combining \eqref{acc-x-hat}
with $\|x_{\beta}\|=\frac{\|x\|}{\beta}$, we have
\begin{eqnarray*}
  \sin\angle(\widehat v,v)&=&\sin\angle(B\widehat x,Bx)
  =\frac{1}{\|Bx\|}\min_{\mu} \| Bx-\mu B \widehat x\| \nonumber \\
  &\leq& \frac{\|B\|}{\|Bx\|} \min_{\mu}\|x-\mu \widehat x\|
   =\frac{\|B\| \|x\| }{\beta}\sin\angle(\widehat x,x)\nonumber  \\
&\leq&
 \|B\|\|x_{\beta}\|   \sqrt{1+\frac{1}{\|x_{\beta}\|^2}}
    \sin\angle\left(\widehat y,y\right)\\
    &=&\|B\|\sqrt{1+\|x_{\beta}\|^2}
    \sin\angle\left(\widehat y,y\right),
\end{eqnarray*}
which proves \eqref{acc-v-hat}.
\end{proof}

As $\|x_{\beta}\|=\frac{\|x\|}{\beta}
\geq\frac{1}{\|B\|}\approx1$, this theorem shows that the
recovered approximate generalized singular vector $\widehat x$ is
unconditionally as accurate as $\widehat y$,
but $\widehat u$ and $\widehat v$
are guaranteed to be as accurate as $\widehat y$ only if
$\beta$ is not small.
As Property~\ref{props} indicates, it is the conditioning of $B$
that determines the size of $\beta$: for $B$ well conditioned,
there is no small $\beta$, so that the recovered approximate
generalized singular vectors are guaranteed to be as accurate
as $\widehat y$; for $B$ ill conditioned, some $\beta$ must be
small so that $\|x_{\beta}\|$ is large,
which correspond to large generalized singular values
$\sigma$, so that the associated recovered $\widehat u$ and
$\widehat v$ may have poorer accuracy than $\widehat y$.

In an analogous manner, we can prove the following results.

\begin{theoremmy}\label{accuracy2}
The scaled right generalized singular vector $x_{\alpha}$,
defined in \eqref{xalbe}, of $(A,B)$ satisfies%
\begin{equation}\label{xalpha}
  \|x_{\alpha}\|\geq \frac{1}{\|A\|}.
\end{equation}
For the approximate generalized singular vectors $\widetilde u,
\widetilde v$ and $\widetilde x$ recovered from the approximate
eigenvector $\widetilde z$ of \eqref{widetildeBA}, it holds that
\begin{eqnarray}
  \sin\angle(\widetilde x,x)&\leq&\sqrt{1+\frac{1}{\|x_{\alpha}\|^2}}
    \sin\angle\left(\widetilde z,z\right),    \label{acc-x-tilde} \\
  \sin\angle(\widetilde v,v)&\leq&\sqrt{1+\|x_{\alpha}\|^2}
    \sin\angle\left(\widetilde z,z\right),    \label{acc-v-tilde} \\
  \sin\angle(\widetilde u,u)&\leq&\|A\|\sqrt{1+\|x_{\alpha}\|^2}
    \sin\angle\left(\widetilde z,z\right).    \label{acc-u-tilde}
\end{eqnarray}
\end{theoremmy}

The comments on Theorem~\ref{accuracy1} apply to this theorem:
$\widetilde x$ is always as accurate as $\widetilde z$;
$\widetilde u$ and $\widetilde v$ are guaranteed to be as accurate as
$\widetilde z$ only if $\alpha$ is fairly modest, and
they may be considerably poorer
than $\widetilde z$ when
$\|x_{\alpha}\|$ is large, i.e., when $\alpha$ is small.
From Property~\ref{props},
it is known that if $A$ is well conditioned then no $\alpha$ is small
but if $A$ is ill conditioned then some $\alpha$ must be small,
which correspond to small generalized singular values $\sigma$.

Recall the previous fundamental conclusions on the accuracy of
$\widehat y$ and $\widetilde z$, which have been summarized
in the near end of Section~\ref{section:2}.
Substituting the bounds in Theorem~\ref{theorem:2} for
$\sin\angle (\widehat y,y)$ and $\sin\angle (\widetilde z,z)$ into
Theorems~\ref{accuracy1}--\ref{accuracy2}, we obtain the
corresponding error estimates for the approximate generalized
singular vectors recovered from the approximate eigenvectors
$\widehat y$ of \eqref{widehatAB} and $\widetilde z$ of \eqref{widetildeBA}, as summarized below.

\begin{theoremmy}
The approximate generalized singular vectors $\widehat u$,
$\widehat v$, and $\widehat x$ recovered from the approximate eigenvector
$\widehat y$ of \eqref{widehatAB} satisfy
\begin{eqnarray}
\sin\angle(\widehat x,x)&\leq&
\frac{\sqrt{1+\frac{1}{\|x_{\beta}\|^2}}\max\{1,\|B^{\dag}\|^2\}}
{\widehat\sigma\min_{\sigma_i\not=\sigma}\{|1-\frac{\sigma_i}{\widehat\sigma}|,1\}}
\sqrt{\|\widehat E\|^2+\widehat\sigma^2\|\widehat F\|^2}, \label{acc-x-hat-sum} \\
\sin\angle(\widehat u,u)&\leq&
\frac{\sqrt{1+\|x_{\beta}\|^2}\max\{1,\|B^{\dag}\|^2\}}
{\widehat\sigma\min_{\sigma_i\not=\sigma}\{|1-\frac{\sigma_i}{\widehat\sigma}|,1\}}
\sqrt{\|\widehat E\|^2+\widehat\sigma^2\|\widehat F\|^2},       \label{acc-u-hat-sum} \\
\sin\angle(\widehat v,v)&\leq&
\frac{\|B\|\sqrt{1+\|x_{\beta}\|^2}\max\{1,\|B^{\dag}\|^2\}}
{\widehat\sigma\min_{\sigma_i\not=\sigma}\{|1-\frac{\sigma_i}{\widehat\sigma}|,1\}}
\sqrt{\|\widehat E\|^2+\widehat\sigma^2\|\widehat F\|^2}.       \label{acc-v-hat-sum}
\end{eqnarray}
Similarly, the approximate generalized singular
vectors $\widetilde u$, $\widetilde v$, and $\widetilde x$ recovered from
the approximate eigenvector $\widetilde z$ of \eqref{widetildeBA} satisfy
\begin{eqnarray}
\sin\angle(\widetilde x,x)&\leq&
\frac{\sqrt{1+\frac{1}{\|x_{\alpha}\|^2}}\max\{1,\|A^{\dag}\|^2\}}
{\min_{\sigma_i\not=\sigma}\{|1-\frac{\widetilde\sigma}{\sigma_i}|,1\}}
\sqrt{\|\widetilde E\|^2+\widetilde\sigma^2\|\widetilde F\|^2},    \label{acc-x-tilde-sum} \\
\sin\angle(\widetilde v,v)&\leq&
\frac{\sqrt{1+\|x_{\alpha}\|^2}\max\{1,\|A^{\dag}\|^2\}}
{\min_{\sigma_i\not=\sigma}\{|1-\frac{\widetilde\sigma}{\sigma_i}|,1\}}
\sqrt{\|\widetilde E\|^2+\widetilde\sigma^2\|\widetilde F\|^2},    \label{acc-v-tilde-sum} \\
\sin\angle(\widetilde u,u)&\leq&
\frac{\|A\|\sqrt{1+\|x_{\alpha}\|^2}\max\{1,\|A^{\dag}\|^2\}}
{\min_{\sigma_i\not=\sigma}\{|1-\frac{\widetilde\sigma}{\sigma_i}|,1\}}
\sqrt{\|\widetilde E\|^2+\widetilde\sigma^2\|\widetilde F\|^2}.    \label{acc-u-tilde-sum}
\end{eqnarray}
\end{theoremmy}

Combining these bounds with the above analysis and
the claims in the near end of Section~\ref{section:2},
we come to the following conclusions on a proper choice of
\eqref{widehatAB} and \eqref{widetildeBA} for more accurate
computations of generalized singular vectors.
\begin{itemize}
\item If both $A$ and $B$ are equally conditioned, i.e,
    both of them are well conditioned, both \eqref{widehatAB}
    and \eqref{widetildeBA} are suitable choices.

\item If $A$ is well conditioned and $B$ is ill conditioned,
    \eqref{widetildeBA} is preferable.

\item If $A$ is ill conditioned and $B$ is well conditioned,
    \eqref{widehatAB} is preferable.
\end{itemize}

By comparing these conclusions with those in the end of Section 2.1
for accurate computations of generalized singular values,
we find out that they exactly coincide. Therefore, we have finally
achieved our ultimate goal of making a proper choice of
\eqref{widehatAB} and \eqref{widetildeBA}: the above conclusions
apply to more accurate computations of both generalized singular
values $\sigma$ and generalized singular vectors $u,v,x$.

\section{Practical choice strategies on
\texorpdfstring{\eqref{widehatAB}}{} and \texorpdfstring{\eqref{widetildeBA}}{}}\label{section:4}

In Sections \ref{section:2}--\ref{section:3} we have made a
sensitivity analysis on the generalized singular values and the
corresponding generalized singular vectors of $(A,B)$, which
are computed by solving the generalized eigenvalue problems of
\eqref{widehatAB} and \eqref{widetildeBA}.
The results have shown that, in order to compute the desired GSVD
components of $(A,B)$ more accurately, we should make a preferable
choice between \eqref{widehatAB} and \eqref{widetildeBA}.
To be practical in computations, this requires to estimate the
condition numbers of $A$ and $B$ efficiently and reliably.

For $A$ and $B$ large-scale, note that we do not need to
estimate $\kappa(A)$ and $\kappa(B)$ accurately,
and rough estimates are enough. Taking $A$ as an
example, we describe three approaches to estimate $\kappa(A)$ roughly.
As $\|A\|\approx1$ and $\kappa(A)\approx\sigma^{-1}_{\min}(A)$,
estimating $\kappa(A)$ is equivalent to estimating $\sigma_{\min}(A)$.

The first approach: if $A$ is large-scale with special structures
such that the matrix-vector multiplication with the matrix $(A^TA)^{-1}$
can be implemented at affordable extra cost, then one can perform a
$k$-step symmetric Lanczos method \cite{baiedit2000,parlett1998symmetric}
on $(A^TA)^{-1}$ and take the square root of the largest approximate
eigenvalue as a reasonable estimate of $\kappa(A)$.
In the algorithm, what we need is to form $A^TA$ and compute its
Cholesky factorization, which is used to solve lower and upper
triangular linear systems at each step of the Lanczos method.
The largest eigenvalue and possibly the smallest eigenvalue of
$(A^TA)^{-1}$ can be well approximated from below and above by
the largest and smallest ones of the symmetric tridiagonal matrices
generated by the Lanczos process, respectively \cite{parlett1998symmetric}.
With $k\ll n$, this method outputs a lower bound for $\kappa(A)$.
Since we do not need to estimate $\kappa(A)$ accurately and the
Lanczos method generally converges quickly for
computing the largest and smallest eigenvalues, we suggest to
take a small $k=20$ in practice.

The second approach: when $A$ is a general large matrix, it is
unaffordable to apply $(A^TA)^{-1}$.
Avron, Druinsky and Toledo~\cite{avron2013spectral} propose a
randomized Krylov subspace method to estimate the condition
number of a matrix $A$.
In their method, a consistent linear least squares problem, whose
solution is generated randomly, is solved iteratively by the LSQR
algorithm \cite{bjorck1996numerical}, and the smallest singular value
of $A$ is estimated by $\sigma_{\min}(A)\approx\frac{\|Ae\|}{\|e\|}$
with $e$ being the error of the approximate solution and the exact one.
We refer the reader to \cite{avron2013spectral} for details.

The third approach: as an alternative of the second approach, one
can also perform a $k$-step Lanczos bidiagonalization type method on $A$
and take the largest and smallest singular values of the resulting
small projected matrix as approximations to the largest and smallest
singular values of $A$; see \cite{jia2003implicitly,jia2010}.
We then take their ratio as a rough approximation to $\kappa(A)$.
Still, we take a small $k=20$ in practice. In this way, we can
efficiently estimate $\kappa(A)$.

Having estimated $\kappa(A)$ and $\kappa(B)$ using one of the above
approaches, taking the resulting estimates as replacements of $\kappa(A)$
and $\kappa(B)$, and based on the previous results and analysis, one
can make a proper choice of \eqref{widehatAB} and \eqref{widetildeBA}
according to the following strategy.
\begin{itemize}
  \item If $0.5\kappa(B)\leq \kappa(A)\leq 2\kappa(B)$, which
      means that $A$ and $B$ are equally well conditioned, then both
      \eqref{widehatAB} and \eqref{widetildeBA} are suitable;
  \item If $\kappa(A)>2\kappa(B)$, which means that $A$ is worse
      conditioned than $B$, then \eqref{widetildeBA} is adopted;
  \item If $\kappa(B)>2\kappa(A)$, which means that $B$ is worse
      conditioned than $A$, then \eqref{widehatAB} is recommended.
\end{itemize}

\section{Numerical experiments}\label{section:6}
In this section, we report numerical experiments to confirm our theory.
We do not aim to develop any algorithms based on
\eqref{widehatAB} and \eqref{widetildeBA} in this paper. Rather, we simply
apply some existing numerically backward stable algorithms to them and
compute their generalized eigendecompositions.
In the experiments, we use the QZ algorithm, i.e.,
the Matlab built-in function {\sf eig}, for the generalized eigenvalue
problems \eqref{widehatAB} and \eqref{widetildeBA}.
For each  matrix pair $(A,B)$, we recover all the
$approximate$ GSVD components
$(\widehat\alpha,\widehat\beta,\widehat u,\widehat v,\widehat x)$ and
$(\widetilde \alpha,\widetilde \beta,\widetilde u,\widetilde v,\widetilde x)$
from the computed eigenpairs of the augmented matrix pairs
$(\widehat A,\widehat B)$ and $(\widetilde B,\widetilde A)$,
respectively, i.e., $(\widehat \sigma,\widehat y)$ and
$(\frac{1}{\widetilde \sigma},\widetilde z)$,
which are obtained by  applying {\sf eig}
to \eqref{widehatAB} and \eqref{widetildeBA}, respectively.
The ``$exact$'' GSVD components  $(\alpha,\beta,u,v,x)$ are computed
by applying the Matlab built-in function {\sf gsvd} to $(A,B)$.
\footnote{For the right generalized singular vector matrix $X$ in
\eqref{GSVD}, {\sf gsvd} outputs $R=X^{-T}$ in our notation.
Hence $X$ is recovered by using the Matlab built-in function {\sf inv}
and taken as the transpose of {\sf inv($R$)}.}

We compare solution accuracy of the GSVD components based on
\eqref{widehatAB} and \eqref{widetildeBA}, and mainly justify
three points: (\romannumeral 1) if both $A$ and $B$ are well conditioned,
then both \eqref{widehatAB} and \eqref{widetildeBA} are suitable
for computing the GSVD of $(A,B)$ accurately;
(\romannumeral 2) if $A$ is ill conditioned and $B$ is well
conditioned, then \eqref{widehatAB} is preferable to compute
the GSVD accurately;
(\romannumeral 3) if $A$ is well conditioned and $B$ is ill
conditioned, then \eqref{widetildeBA} is a better formulation
for computing the GSVD accurately.
As mentioned in the beginning of section 1, the GSVDs of the matrix
pairs $(A,B)$ and $(B,A)$ are the same with the generalized singular
values being the reciprocals of each other.
Under the assumption that at least one of $A$ and $B$ is well
conditioned, we can always take one of them to be well conditioned
and the other one well conditioned or ill conditioned.
Therefore, for the sake of certainty in the experiments,
we always take $B$ to be
well conditioned but $A$ to be well or ill conditioned.
In the meantime, we justify Property~\ref{props}.

All the numerical experiments were performed on an Intel (R) Core
(TM) i7-7700 CPU 3.60 GHz with 8 GB RAM, 4 cores
and 8 threads using the Matlab R2017a with the machine precision
$\epsilon_{\rm mach} =2.22\times10^{-16}$ under the Microsoft
Windows 8 64-bit system.

We measure the accuracy of the computed generalized singular
values by their chordal distances from their exact counterparts and
measure the accuracy of the computed generalized singular vectors by
the sines of the angles between them and their exact counterparts.

Each figure in this section consists of four subfigures: the top left
one depicts the accuracy of the computed generalized singular
values $\sigma_i$ such that $\sigma$'s are sorted in descending order;
the top right, bottom left and right ones depict the accuracy of the
computed right and left generalized singular vectors $x_i$ and $u_i$, $v_i$, respectively.

\begin{exper}
We first test three randomly generated problems.
For prescribed constants $c_A\geq 1$ and $c_B\geq1$, we generate
the random sparse $m\times n$ matrix $A$ and $p\times n$ matrix
$B$ by the Matlab commands
\begin{displaymath}
  A={\sf sprand}(m,n,dens,ra)
  \quad\mbox{and}\quad
  B={\sf sprand}(p,n,dens,rb)
\end{displaymath}
with the density $dens=50\%$, and
$ra=[\tfrac{1}{c_A}:\tfrac{1}{n-1}(1-\tfrac{1}{c_A}):1]$
and $rb=[\tfrac{1}{c_B}:\tfrac{1}{n-1}(1-\tfrac{1}{c_B}):1]$.
The largest singular values of such $A$ and $B$ are equal to one,
i.e., $\|A\|=\|B\|=1$,
and their condition numbers are $c_A$ and $c_B$, respectively.
Therefore, by prescribing the values of $c_A$ and $c_B$,
we control the condition numbers $\kappa(A)$ and $\kappa(B)$.
Table \ref{table1} lists the test problems together with their basic
properties. Figures \ref{fig1}-\ref{fig3} display the results.
\end{exper}

\begin{table}[tbhp]
{\small
\caption{Properties of the test problems with $m=1500$,
$p=2000$ and $n=1000$. }\label{table1}
\begin{center}
\begin{tabular}{|c|c|c|c|c|c|c|c|} \hline
{\rm Problem}&$\kappa(A)$&$\kappa(B)$
&$\kappa(\begin{bmatrix}\begin{smallmatrix}A\\B\end{smallmatrix})\end{bmatrix}$
&$\|X\|$ & $\|X^{-1}\|$&$\sigma_{\max}(A,B)$&$\sigma_{\min}(A,B)$\\ \hline
{1a}&$1.0e+2$&$1.0e+2$&$7.03$&$5.31$&$1.32$&$58.1$&$1.57e-2$\\
{1b}&$1.0e+5$&$1.0e+2$&$5.84$&$4.45$&$1.31$&$65.0$&$2.04e-5$\\
{1c}&$1.0e+7$&$1.0e+2$&$9.55$&$7.19$&$1.33$&$65.1$&$1.70e-7$\\ \hline
\end{tabular}
\end{center}
}
\end{table}

\begin{figure}[htbp]
\centering
\begin{minipage}{1\textwidth}
\subfloat[$\mathcal{X}(\widehat\sigma,\sigma)$ and $
\mathcal{X}(\widetilde\sigma,\sigma)$]
{\label{fig1a}\includegraphics[width=0.48\textwidth]{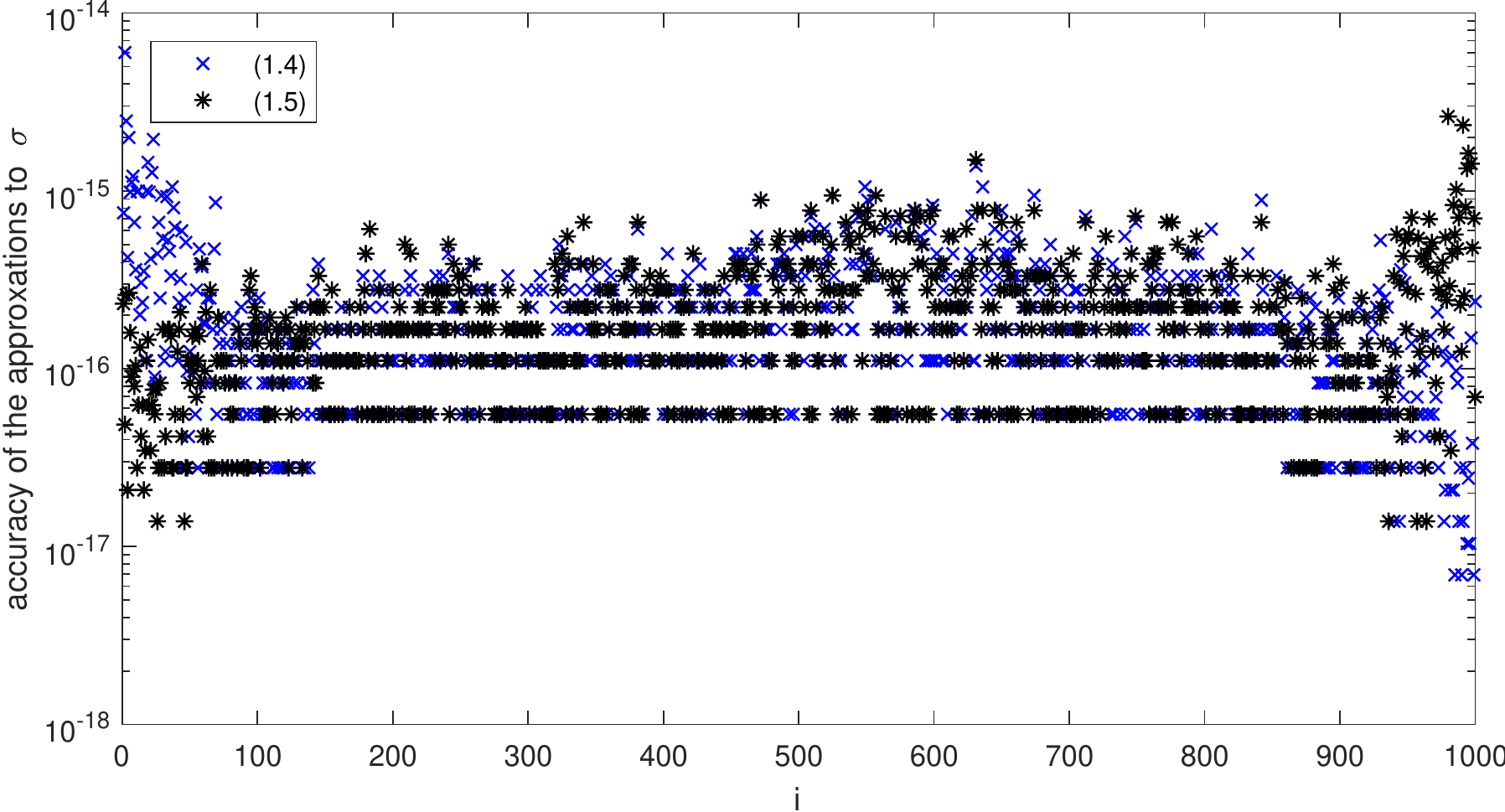}}
\quad
\subfloat[$\sin\angle(\widehat x,x)$ and
$\sin\angle(\widetilde x,x)$]
{\label{fig1b}\includegraphics[width=0.48\textwidth]{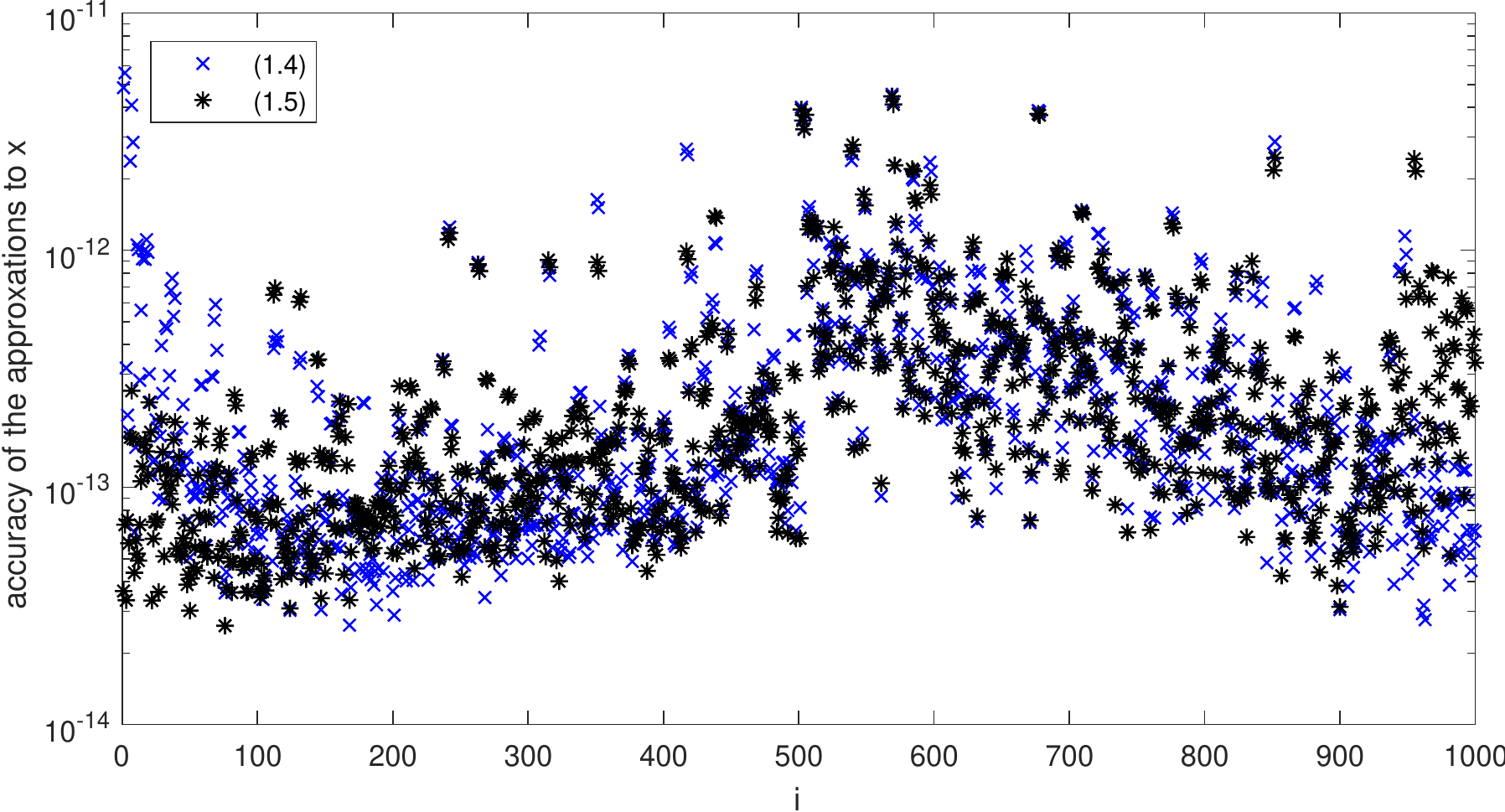}}\\
\subfloat[$\sin\angle(\widehat u,u)$ and
$\sin\angle(\widetilde u,u)$]
{\label{fig1c}\includegraphics[width=0.48\textwidth]{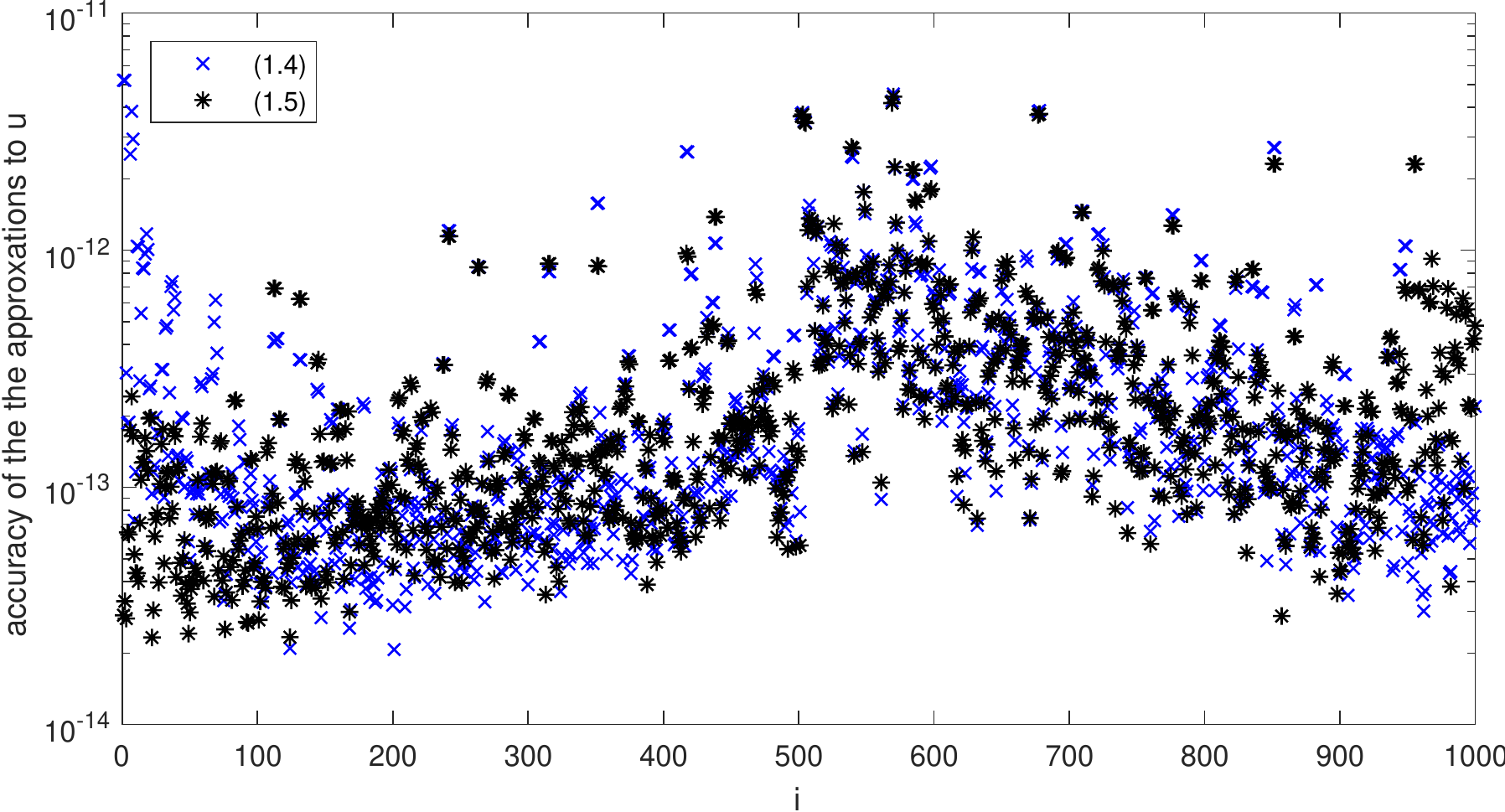}}
\quad
\subfloat[$\sin\angle(\widehat v,v)$ and
$\sin\angle(\widetilde v,v)$]
{\label{fig1d}\includegraphics[width=0.48\textwidth]{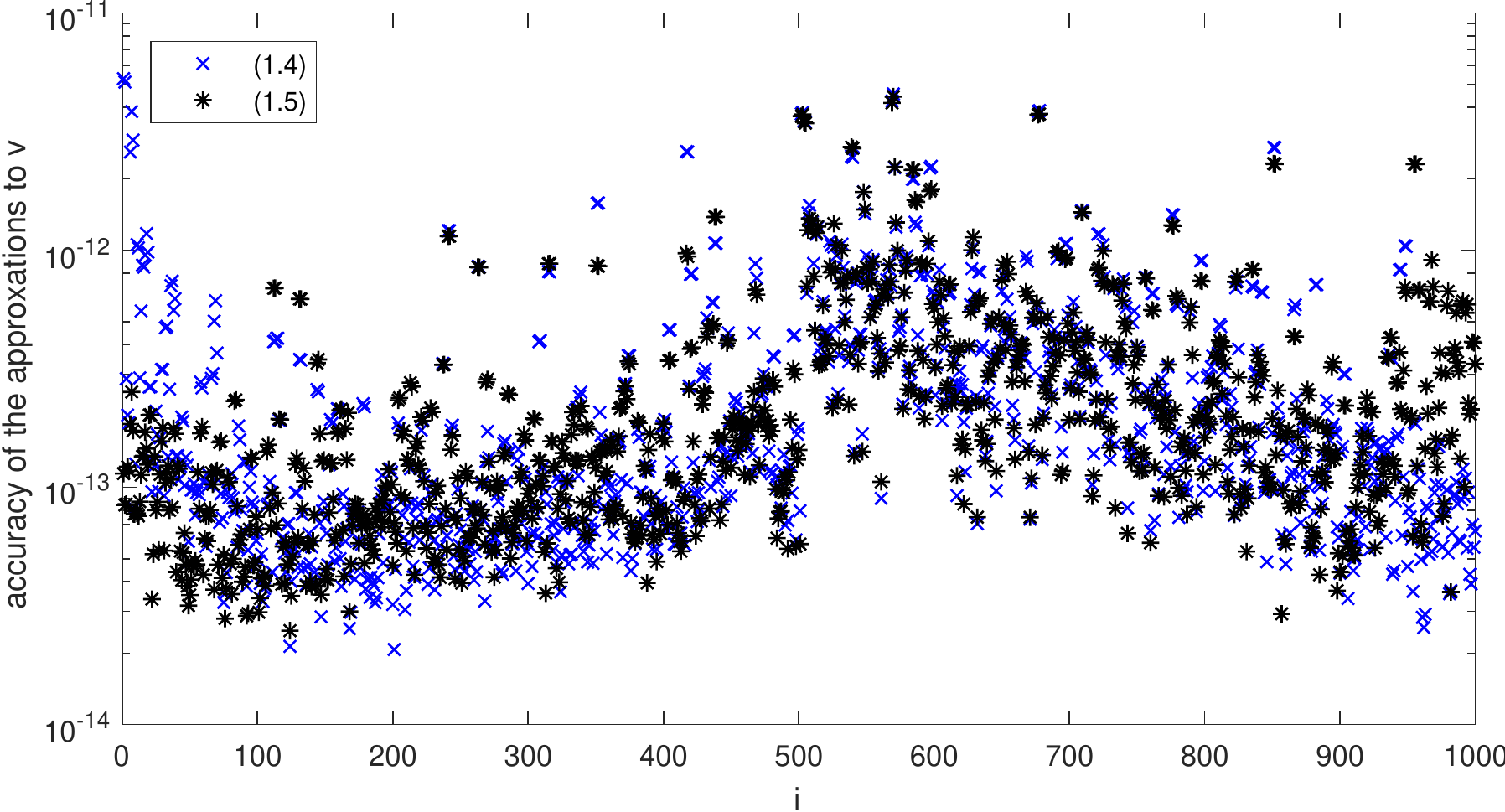}}
\caption{Accuracy of the GSVD components of problem 1a.}\label{fig1}
\end{minipage}
\end{figure}

\begin{figure}[htbp]
\centering
\begin{minipage}{1\textwidth}
\subfloat[$\mathcal{X}(\widehat\sigma,\sigma)$ and
$\mathcal{X}(\widetilde\sigma,\sigma)$]
{\label{fig2a}\includegraphics[width=0.49\textwidth]{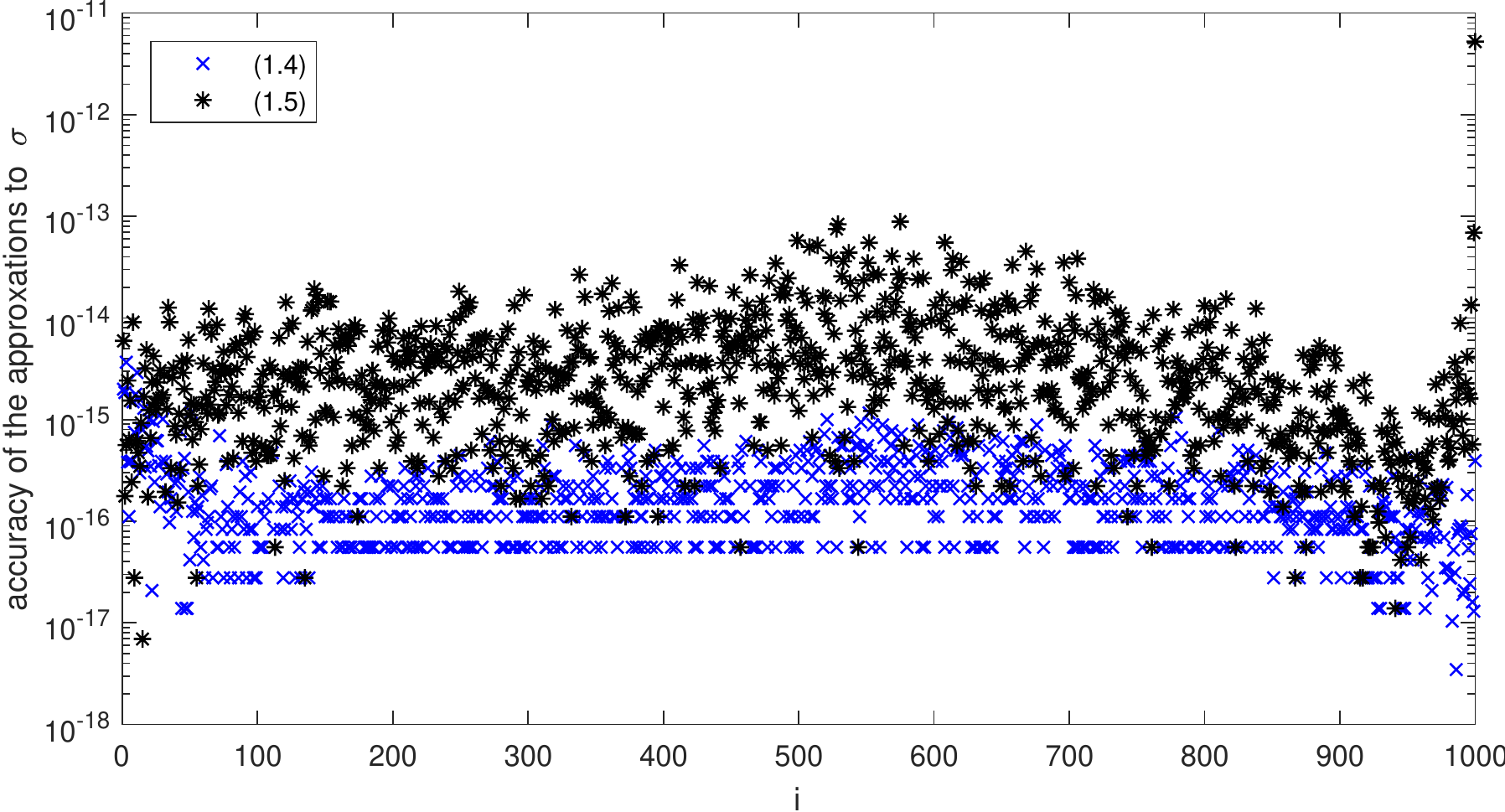}}
\ \ \ \
\subfloat[$\sin\angle(\widehat x,x)$ and
$\sin\angle(\widetilde x,x)$]
{\label{fig2b}\includegraphics[width=0.49\textwidth]{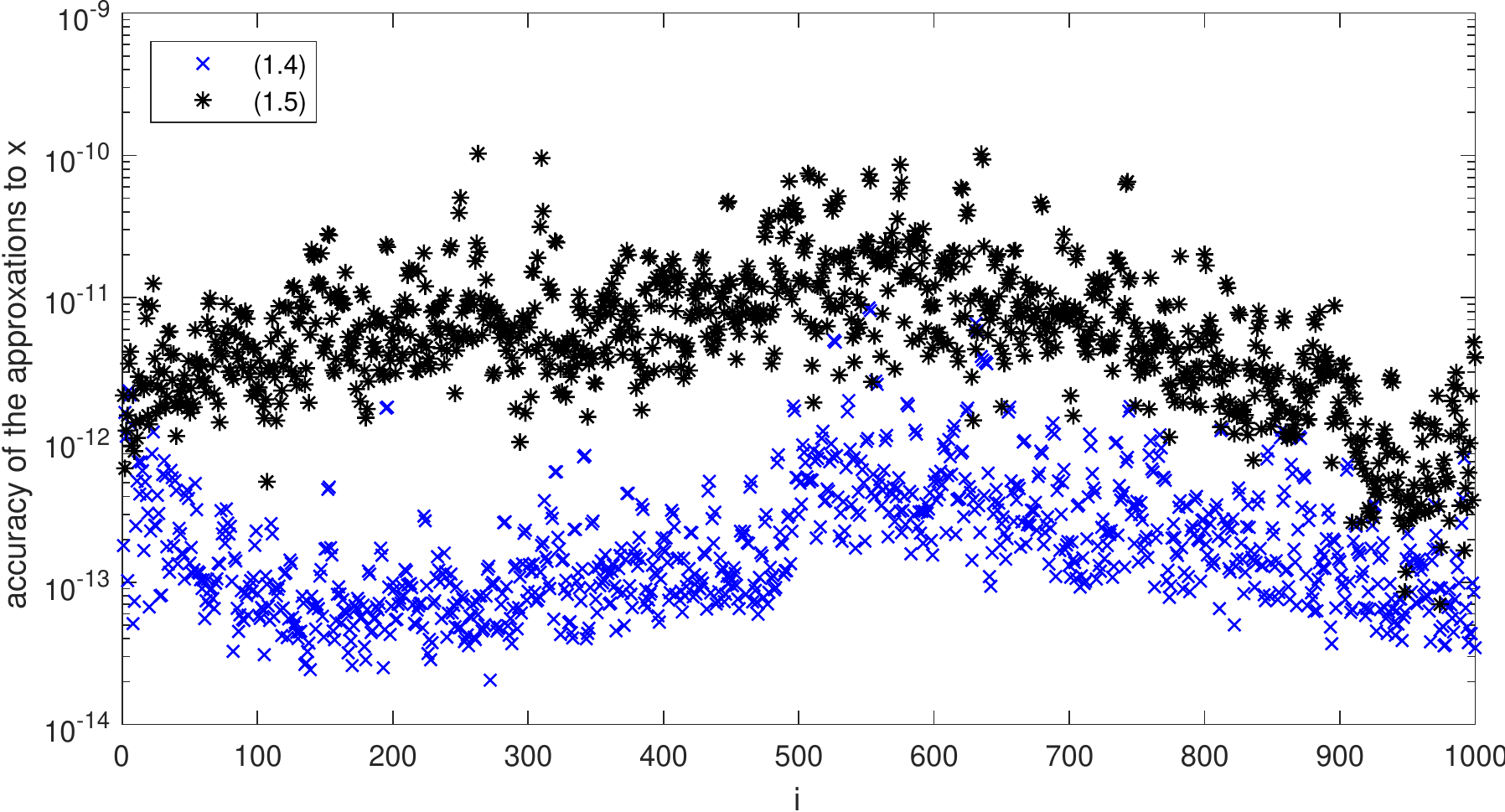}}\\
\subfloat[$\sin\angle(\widehat u,u)$ and
$\sin\angle(\widetilde u,u)$]
{\label{fig2c}\includegraphics[width=0.49\textwidth]{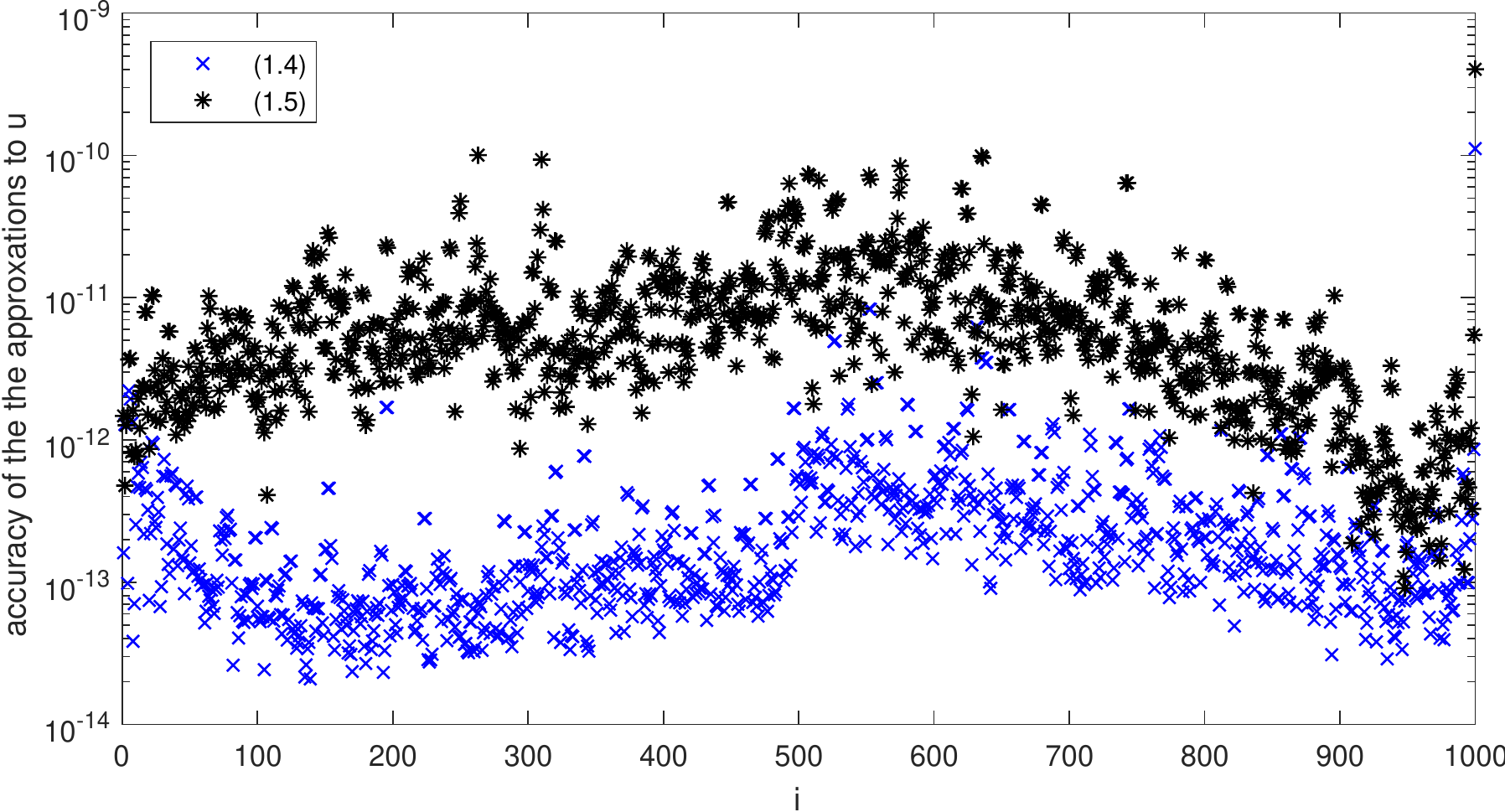}}
\ \ \ \
\subfloat[$\sin\angle(\widehat v,v)$ and
$\sin\angle(\widetilde v,v)$]
{\label{fig2d}\includegraphics[width=0.49\textwidth]{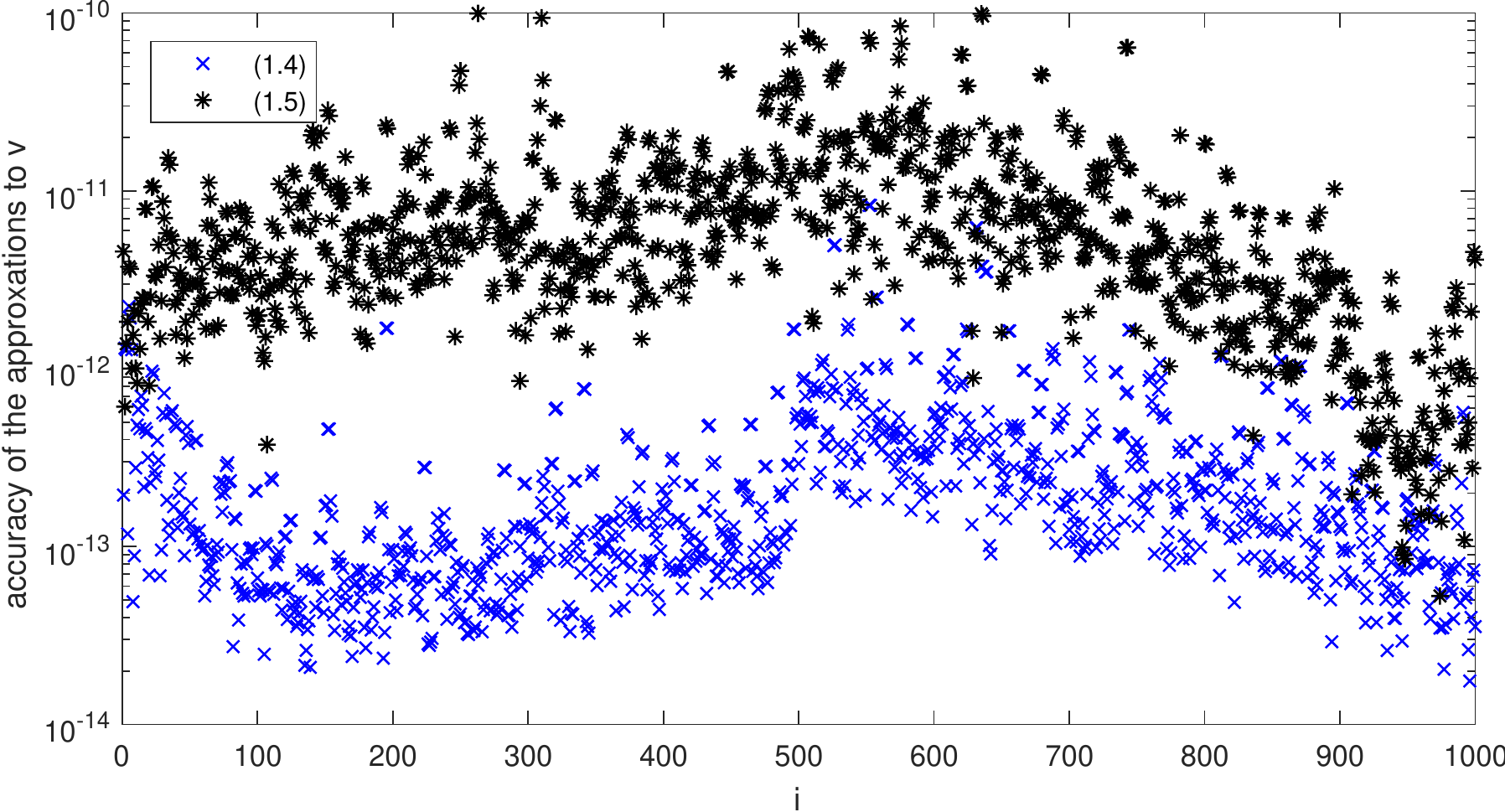}}
\caption{Accuracy of the GSVD components of problem 1b.}\label{fig2}
\vspace{0.5cm}
\end{minipage}
\begin{minipage}{1\textwidth}
\subfloat[$\mathcal{X}(\widehat\sigma,\sigma)$ and
$\mathcal{X}(\widetilde\sigma,\sigma)$]
{\label{fig3a}\includegraphics[width=0.49\textwidth]{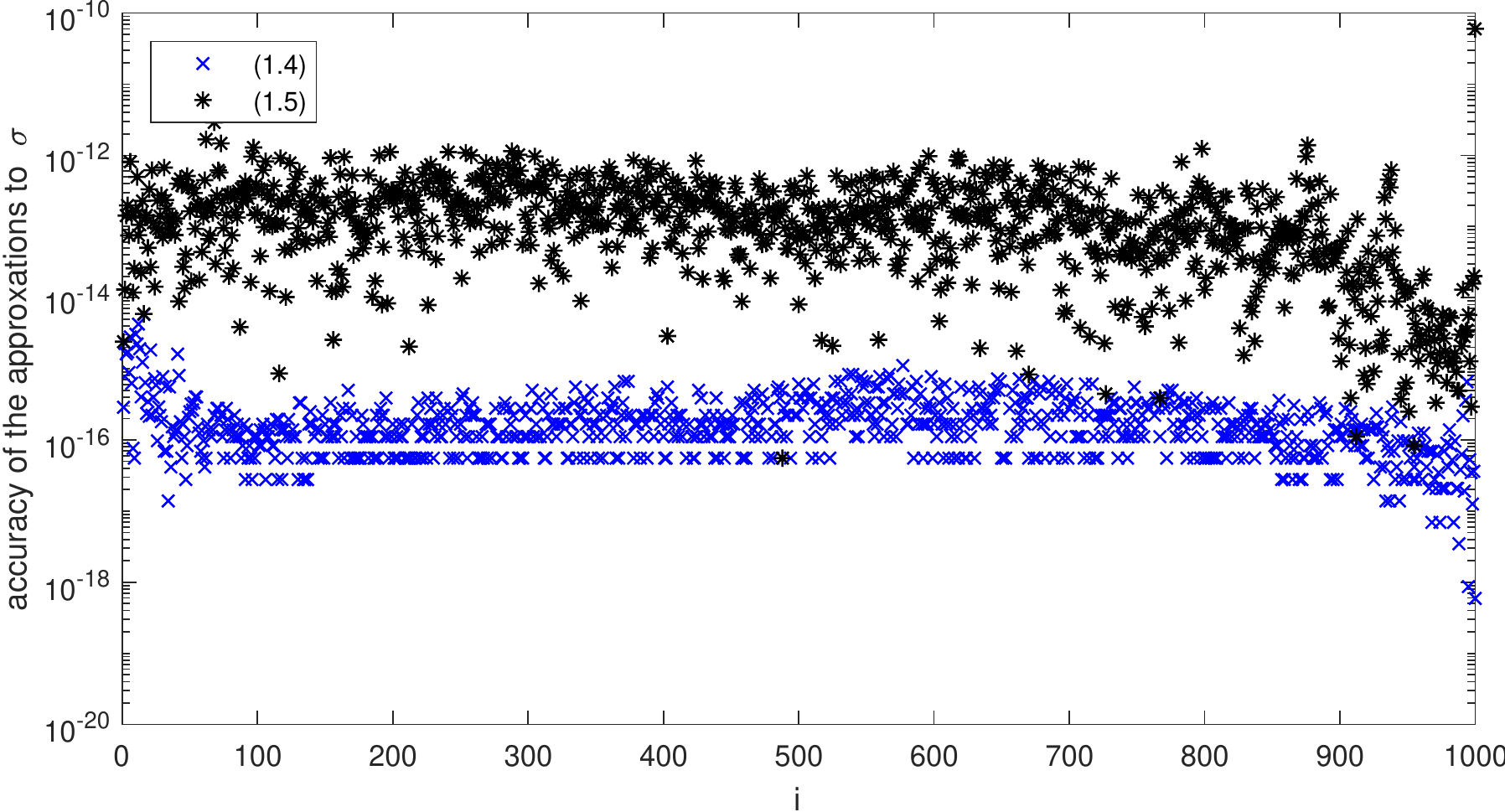}}
\ \ \ \
\subfloat[$\sin\angle(\widehat x,x)$ and
$\sin\angle(\widetilde x,x)$]
{\label{fig3b}\includegraphics[width=0.49\textwidth]{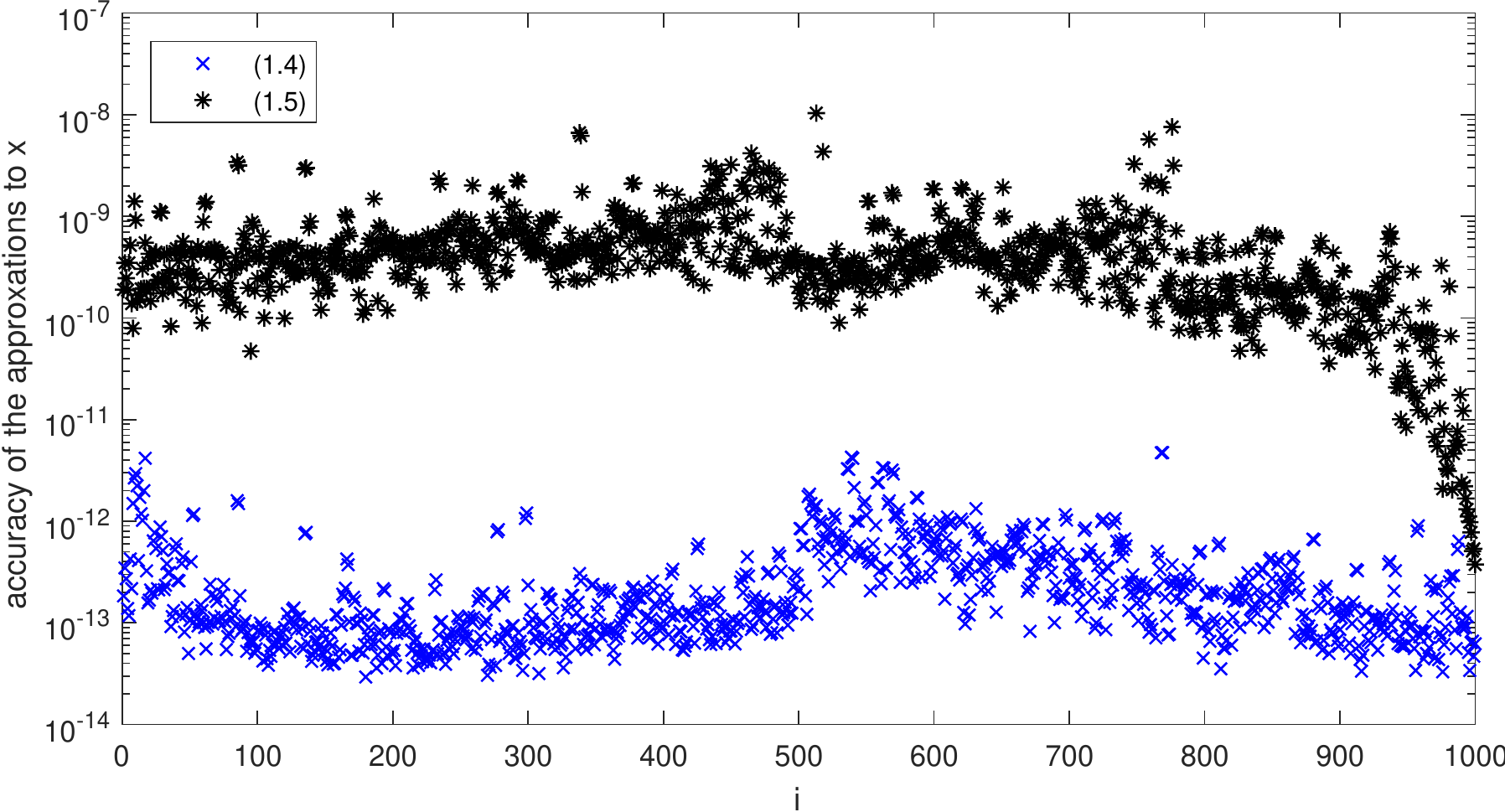}}\\
\subfloat[$\sin\angle(\widehat u,u)$ and
$\sin\angle(\widetilde u,u)$]
{\label{fig3c}\includegraphics[width=0.49\textwidth]{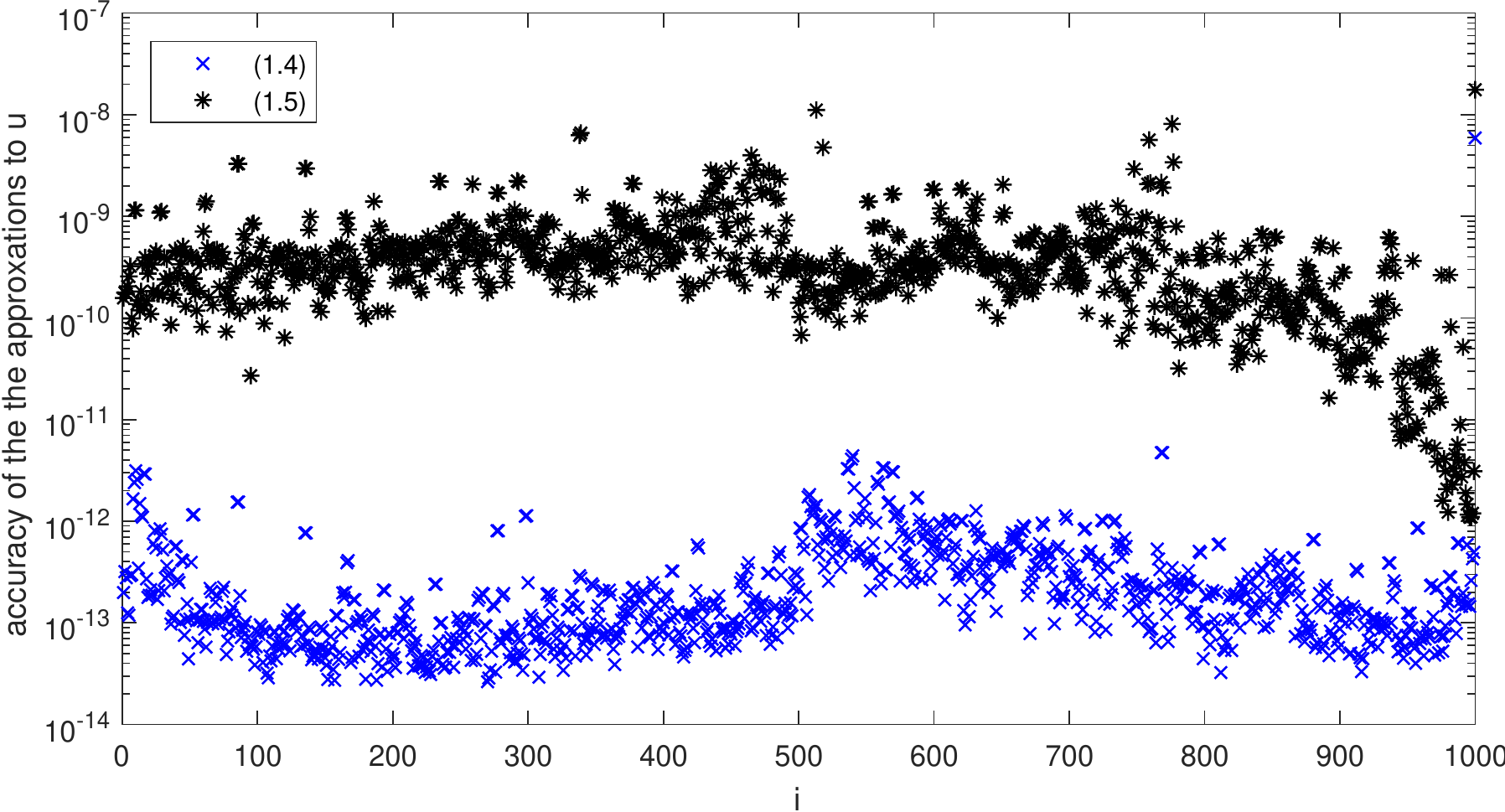}}
\ \ \ \
\subfloat[$\sin\angle(\widehat v,v)$ and
$\sin\angle(\widetilde v,v)$]
{\label{fig3d}\includegraphics[width=0.49\textwidth]{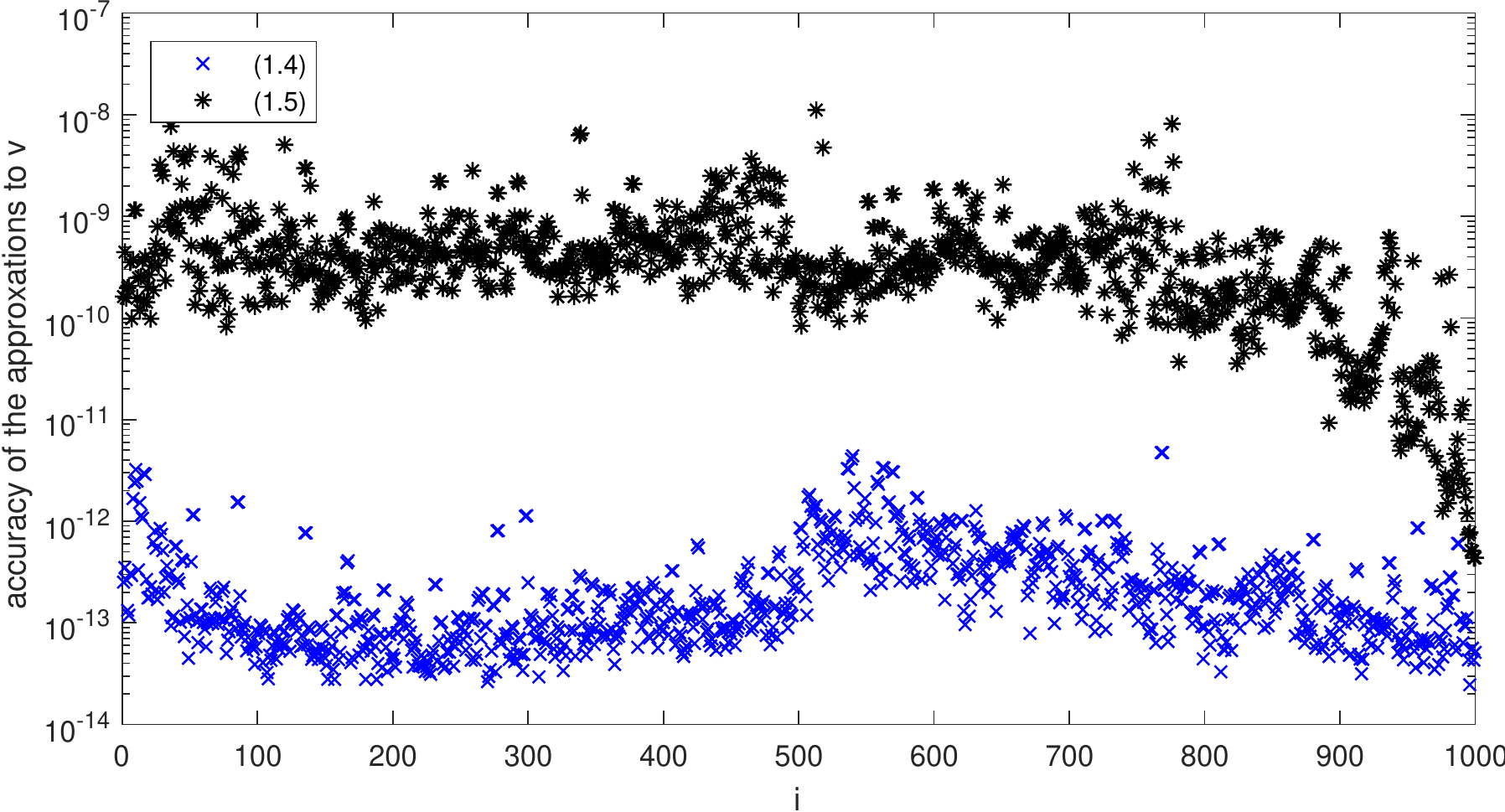}}
\caption{Accuracy of the GSVD components of problem 1c.}\label{fig3}
\end{minipage}
\end{figure}

From Table~\ref{table1} we see that
$\kappa(A)=\|A^{\dagger}\|\approx \frac{1}{\sigma_{\min}(A,B)}$ and
$\kappa(B)=\|B^{\dagger}\|\approx \sigma_{\max}(A,B)$,
confirming Property~\ref{props}
and the third conclusion in the near end of Section 2.
We notice that as long as at least one of $A$ and $B$ is well
conditioned, so is the stacked matrix
$\begin{bmatrix}\begin{smallmatrix}A\\B\end{smallmatrix}\end{bmatrix}$.

For problem 1a, both $A$ and $B$ are well conditioned.
Figure \ref{fig1} illustrates that both \eqref{widehatAB} and
\eqref{widetildeBA} yield equally accurate GSVD components of $(A,B)$.
Apparently, there is no winner between \eqref{widehatAB} and
\eqref{widetildeBA} for this problem.

For problem 1b, $A$ is moderately ill conditioned and $B$ is well conditioned.
As is observed from Figure \ref{fig2a}, the computed generalized
singular values based on \eqref{widehatAB} are generally more
accurate than those based on \eqref{widetildeBA}, or at least as
comparably accurate as the latter ones.
Figures \ref{fig2b}-\ref{fig2d} show that for most
of the generalized singular vectors,
\eqref{widehatAB} yields significantly more accurate approximations
than \eqref{widetildeBA} does.
Therefore, \eqref{widehatAB} outperforms \eqref{widetildeBA} for this problem.

For problem 1c where $A$ is quite ill conditioned and $B$ is well
conditioned, the advantage of \eqref{widehatAB} over
\eqref{widetildeBA} is very obvious.
As is visually illustrated by Figure \ref{fig3},
for all the generalized singular components,
\eqref{widehatAB} yields more or even much
more accurate approximations than \eqref{widetildeBA},
and the accuracy is improved by several orders.
For this problem, \eqref{widehatAB} definitely wins.

For these three problems, we have observed that for both $A$ and
$B$ well conditioned, two formulations \eqref{widehatAB} and
\eqref{widetildeBA} based backward stable algorithms deliver equally
accurate approximations to the GSVD components of $(A,B)$.
For the problems where $A$ is ill conditioned and $B$ is well
conditioned, \eqref{widehatAB} can produce more and even
much more accurate GSVD components than \eqref{widetildeBA}.
Moreover, with $B$ being well conditioned, the worse conditioned
$A$ is, the more advantageous \eqref{widehatAB} is over \eqref{widetildeBA}.
As is also observed from Figures \ref{fig1}-\ref{fig3}, a suitable
choice between \eqref{widehatAB} and \eqref{widetildeBA} can
always guarantee that under the chordal measure all the
generalized singular values $\sigma$ can be computed
with full accuracy, i.e., the level of $\epsilon_{\rm mach}$,
which confirms Theorem~\ref{thm:1} and the analysis followed in Section 2.1.

\begin{exper}\label{exam2}
We test several realistic problems.
For each problem, the matrices $A$ and $B$ are normalized from
$A_0$ and $B_0$, respectively, i.e., $A=\frac{A_0}{\|A_0\|}$ and
$B=\frac{B_0}{\|B_0\|}$, where $A_0\in\mathbb{R}^{n\times n}$
is a square matrix from the SuiteSparse Matrix Collection \cite{davis2011university}
and
\begin{displaymath}
B_0=\begin{bmatrix}1&&\\-1&\ddots&\\&\ddots&1\\&&-1
\end{bmatrix}\in\mathbb{R}^{{(n+1)}\times n}
\end{displaymath}
is the transpose of the $n\times (n+1)$ first order derivative
operator in dimension one~\cite{hansen1998rank}.
Table~\ref{table2} lists the test problems together with some
of their basic properties, where the names inside the brackets
are those of the initial matrices $A_0$,  in which ``delan12'' and
``viscopl1'' are abbreviations for ``delaunay\_n12'' and ``viscoplastic1'',
respectively.
\end{exper}

\begin{table}[tbhp]
{\small\caption{Properties of the test problems with
$m=n$ and $p=n+1$.}\label{table2}
\begin{center}\begin{tabular}{|p{2.10cm}|c|c|c|c|c|c|} \hline
\centering{Problem} &$n$ &$\kappa(A)$ &$\kappa(B)$
&$\kappa(\begin{bmatrix}\begin{smallmatrix}A\\B
\end{smallmatrix}\end{bmatrix})$
&$\sigma_{\max}(A,B)$&$\sigma_{\min}(A,B)$\\ \hline
{2a (3elt)  }
&$4720$&$2.8e+3$&$3.0e+3$&$6.35$   &$2.89e+3$&$5.00e-4$   \\
{2b (\rm delan12)}
&$4096$&$5.2e+3$&$2.6e+3$&$5.15$ &$2.35e+3$&$2.65e-4$   \\
{2c (viscopl1)}
&$4326$&$1.4e+5$&$2.8e+3$&$468$&$1.53e+3$&$9.34e-6$   \\
{2d (cavity16)}
&$4562$&$9.4e+6$&$2.9e+3$&$75.6$   &$1.23e+2$&$1.51e-7$   \\
{2e (gemat11)}
&$4929$&$6.0e+7$&$3.1e+3$&$512$&$23.8$   &$2.65e-8$   \\
{2f (bcsstk16)}&$4884$ &$4.9e+9$&$3.1e+3$&$78.7$&$73.5$&$2.86e-10$\\
\hline
\end{tabular}\end{center}}\end{table}

We observe from Table~\ref{table2} that
$\kappa(A)\approx \frac{1}{\sigma_{\min}(A,B)}$ well and
$\kappa(B)\approx \sigma_{\max}(A,B)$ roughly,
justifying Property~\ref{props} and
the third conclusion in the near end of Section 2.

\begin{table}[tbhp]
{\small\caption{A comparison of \eqref{widehatAB} and
\eqref{widetildeBA} for computing the GSVDs
of test problems 2a-2f.}\label{table3}
\begin{center}\begin{tabular}{|c|c|c|c|c|c|c|c|c|} \hline
\multirow{2}{1.2cm}{Problem}
&\multicolumn{2}{c|}{better $\sigma$}
&\multicolumn{2}{c|}{better $x$}
&\multicolumn{2}{c|}{better $u$}
&\multicolumn{2}{c|}{better $v$} \\ \cline{2-9}
&$pct$($\%$)&$acc$ &$pct$($\%$)&$acc$&$pct$($\%$)&$acc$&$pct$($\%$)&$acc$\\ \hline
{2a}&$43.37$&$-0.24$&$35.68$&$-0.12$&$35.93$&$-0.12$&$35.91$&$-0.12$\\
{2b}&$47.22$&$-0.11$&$36.45$&$-0.07$&$36.62$&$-0.07$&$36.57$&$-0.07$\\
{2c}&$83.93$&$+0.89$&$83.38$&$+0.67$&$84.51$&$+0.71$&$84.26$&$+0.71$\\
{2d}&$85.60$&$+2.26$&$79.35$&$+2.05$&$79.37$&$+2.04$&$79.35$&$+2.04$\\
{2e}&$86.61$&$+1.00$&$97.28$&$+1.12$&$95.94$&$+1.05$&$95.94$&$+1.05$\\
{2f}&$99.20$&$+6.60$&$99.20$&$+6.33$&$99.20$&$+6.34$&$99.20$&$+6.34$\\
\hline
\end{tabular}\end{center}}\end{table}

\begin{figure}[tbhp]
\centering
\begin{minipage}{1\textwidth}
 \subfloat[$\mathcal{X}(\widehat\sigma,\sigma)$ and $
\mathcal{X}(\widetilde\sigma,\sigma)$]
{\label{fig4a}\includegraphics[width=0.49\textwidth]{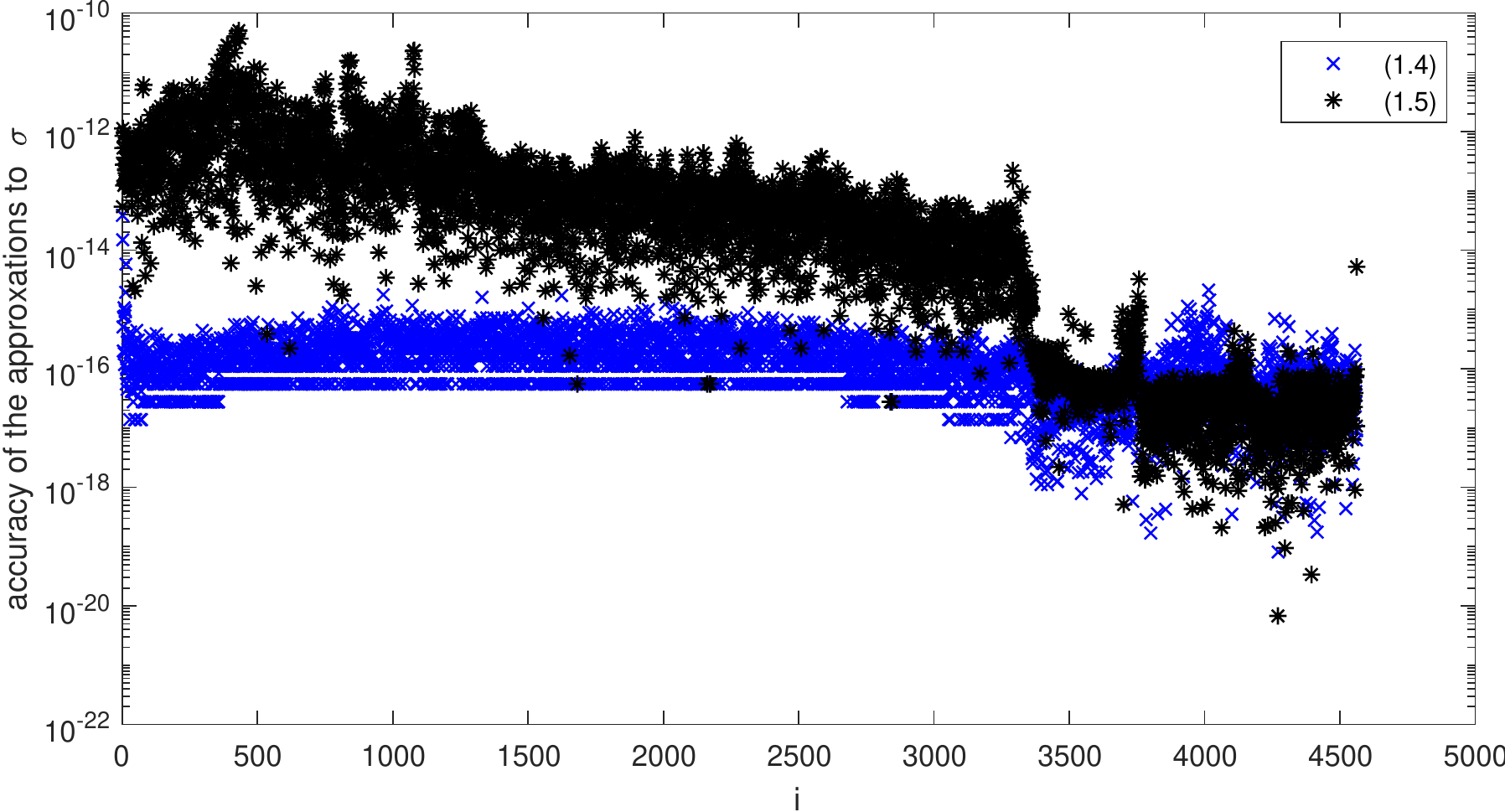}}
\ \ \ \
\subfloat[$\sin\angle(\widehat x,x)$ and
$\sin\angle(\widetilde x,x)$ ]
{\label{fig4b}\includegraphics[width=0.49\textwidth]{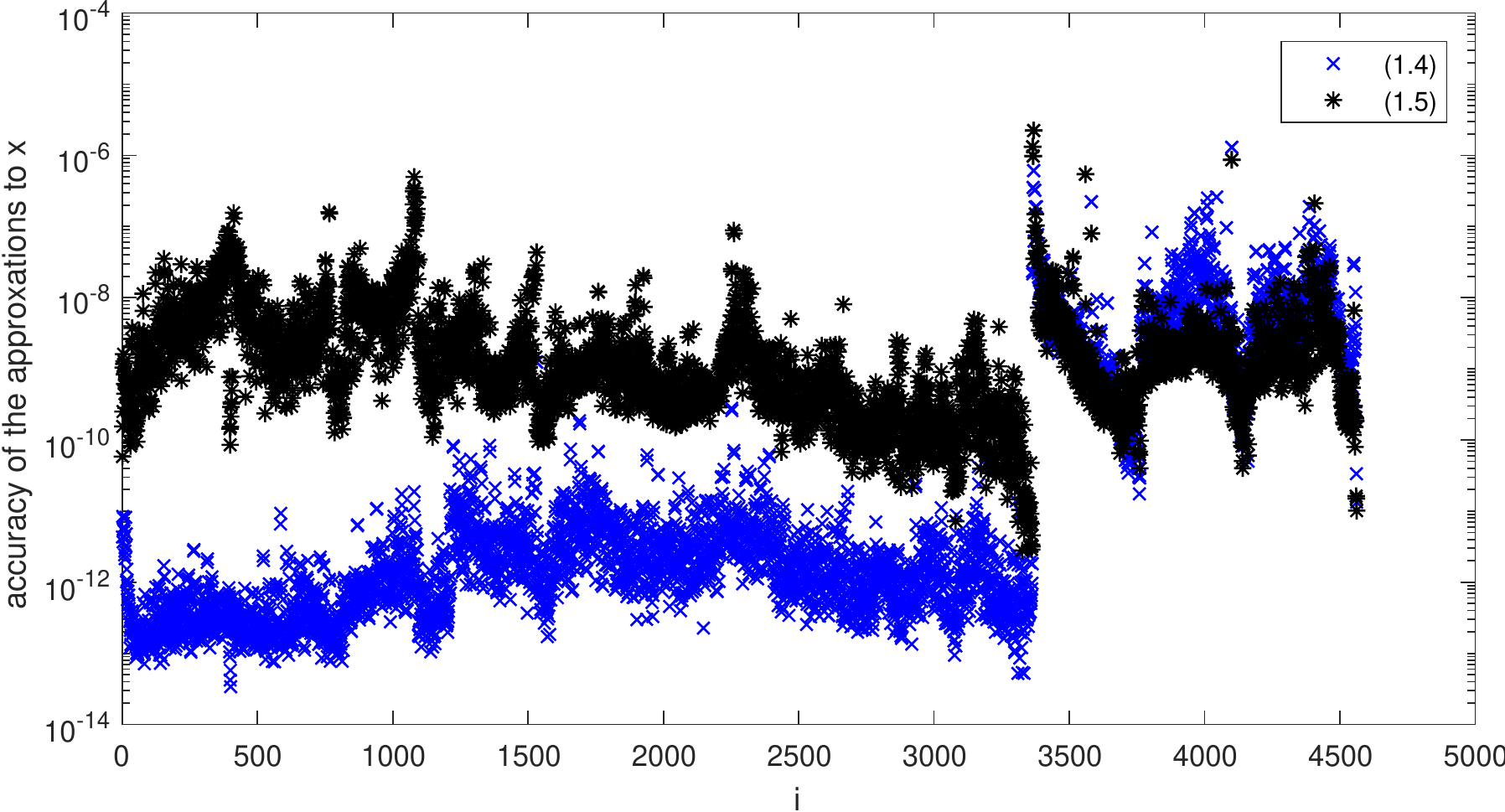}}\\
\subfloat[$\sin\angle(\widehat u,u)$ and
$\sin\angle(\widetilde u,u)$ ]
{\label{fig4c}\includegraphics[width=0.49\textwidth]{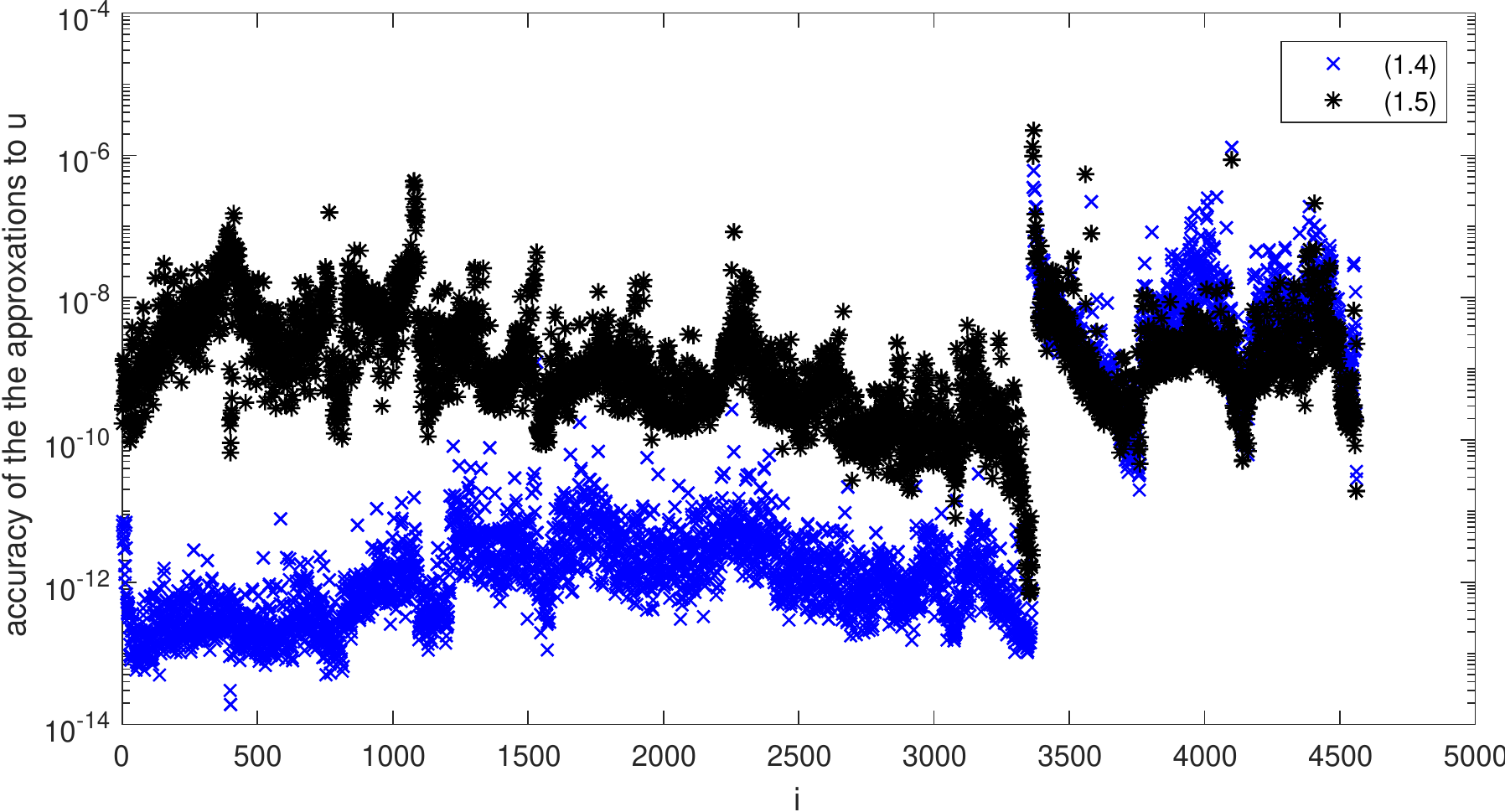}}
\ \ \ \
\subfloat[$\sin\angle(\widehat v,v)$ and
$\sin\angle(\widetilde v,v)$  ]
{\label{fig4d}\includegraphics[width=0.49\textwidth]{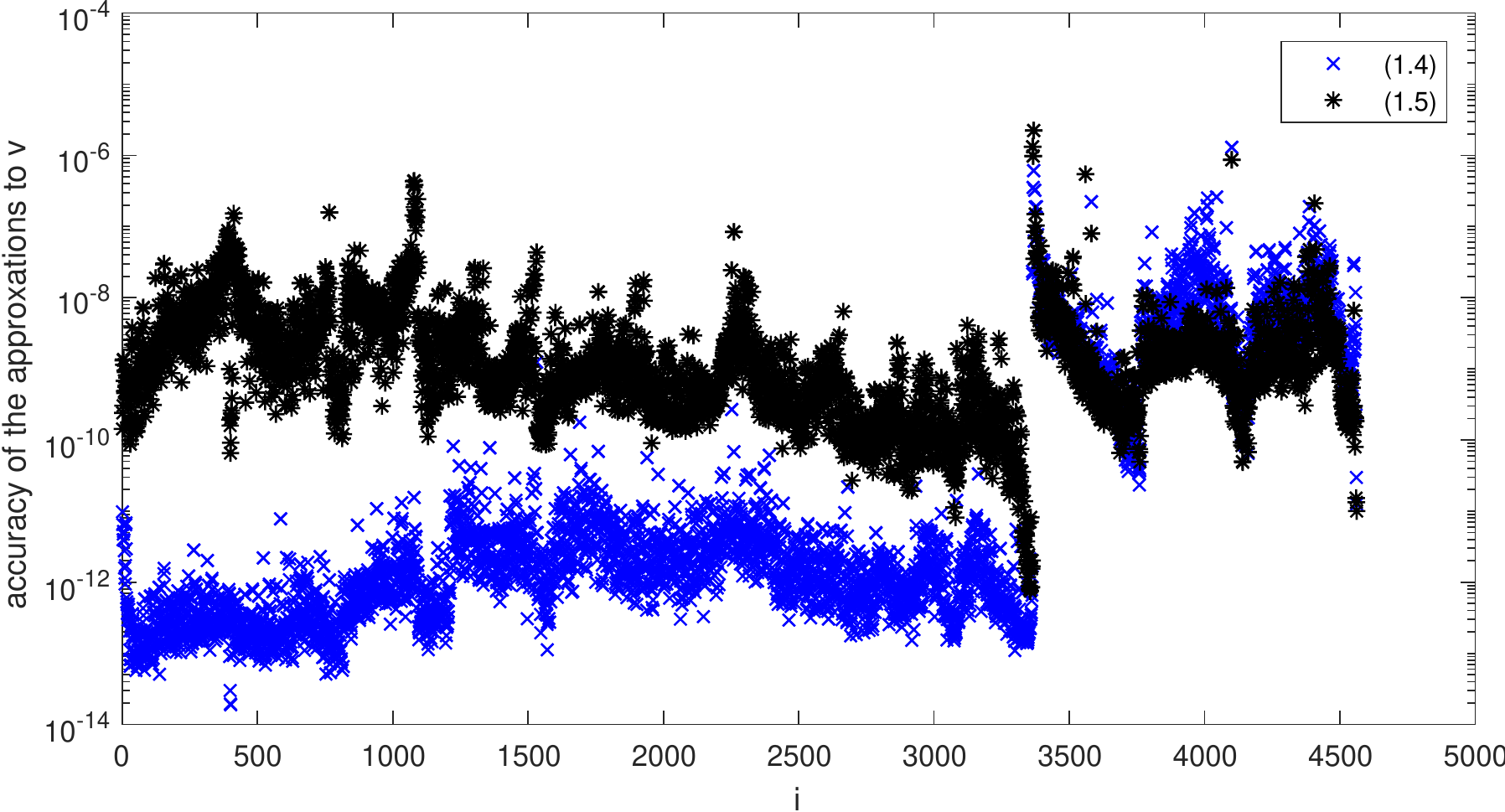}}
\caption{Accuracy of the GSVD components of problem 2d.}\label{fig4}
\vspace{0.5cm}
\end{minipage}
\begin{minipage}{1\textwidth}
\subfloat[$\mathcal{X}(\widehat\sigma,\sigma)$ and $
\mathcal{X}(\widetilde\sigma,\sigma)$]
{\label{fig5a}\includegraphics[width=0.49\textwidth]{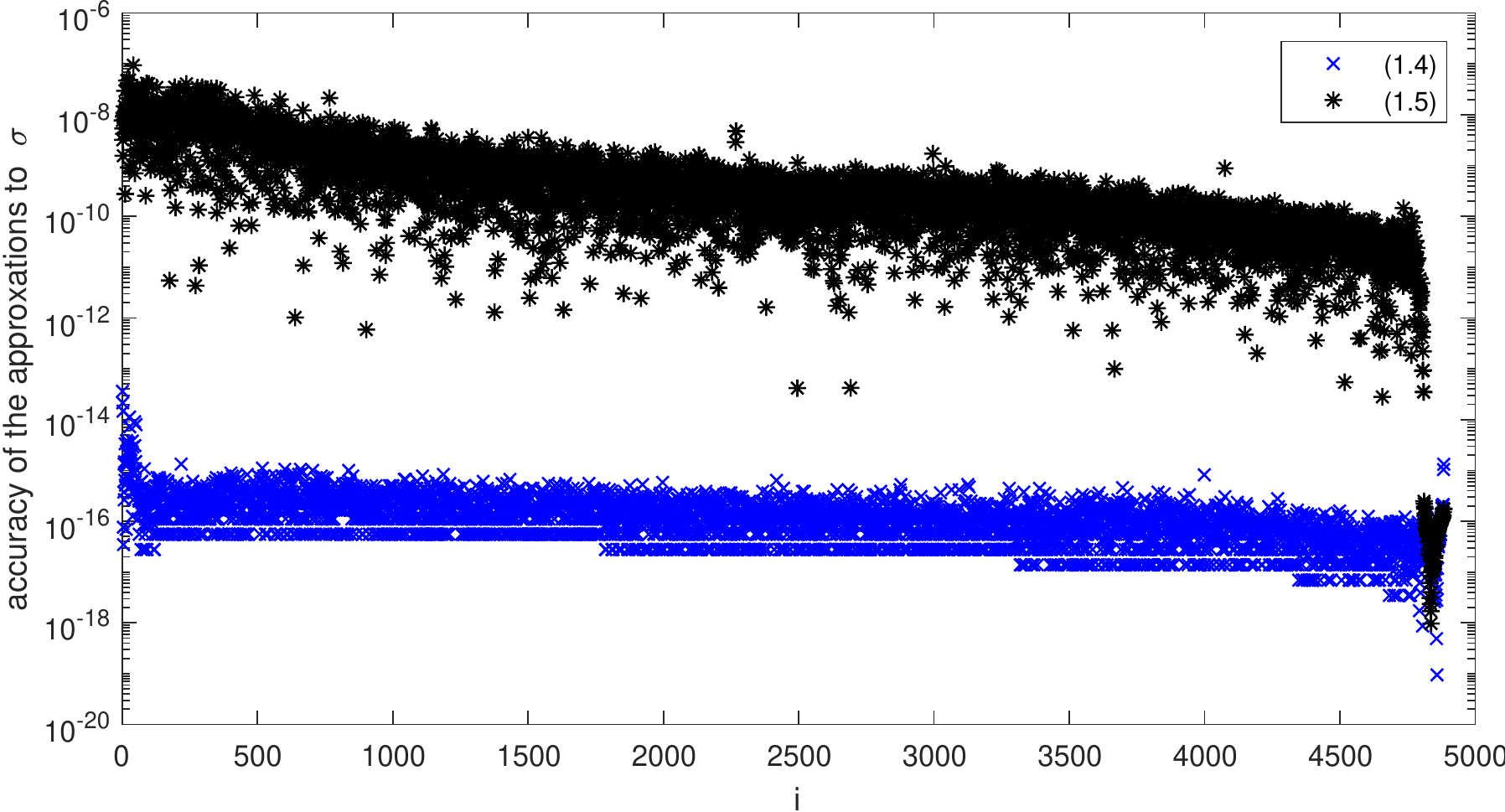}}
\ \ \ \
\subfloat[$\sin\angle(\widehat x,x)$ and
$\sin\angle(\widetilde x,x)$ ]
{\label{fig5b}\includegraphics[width=0.49\textwidth]{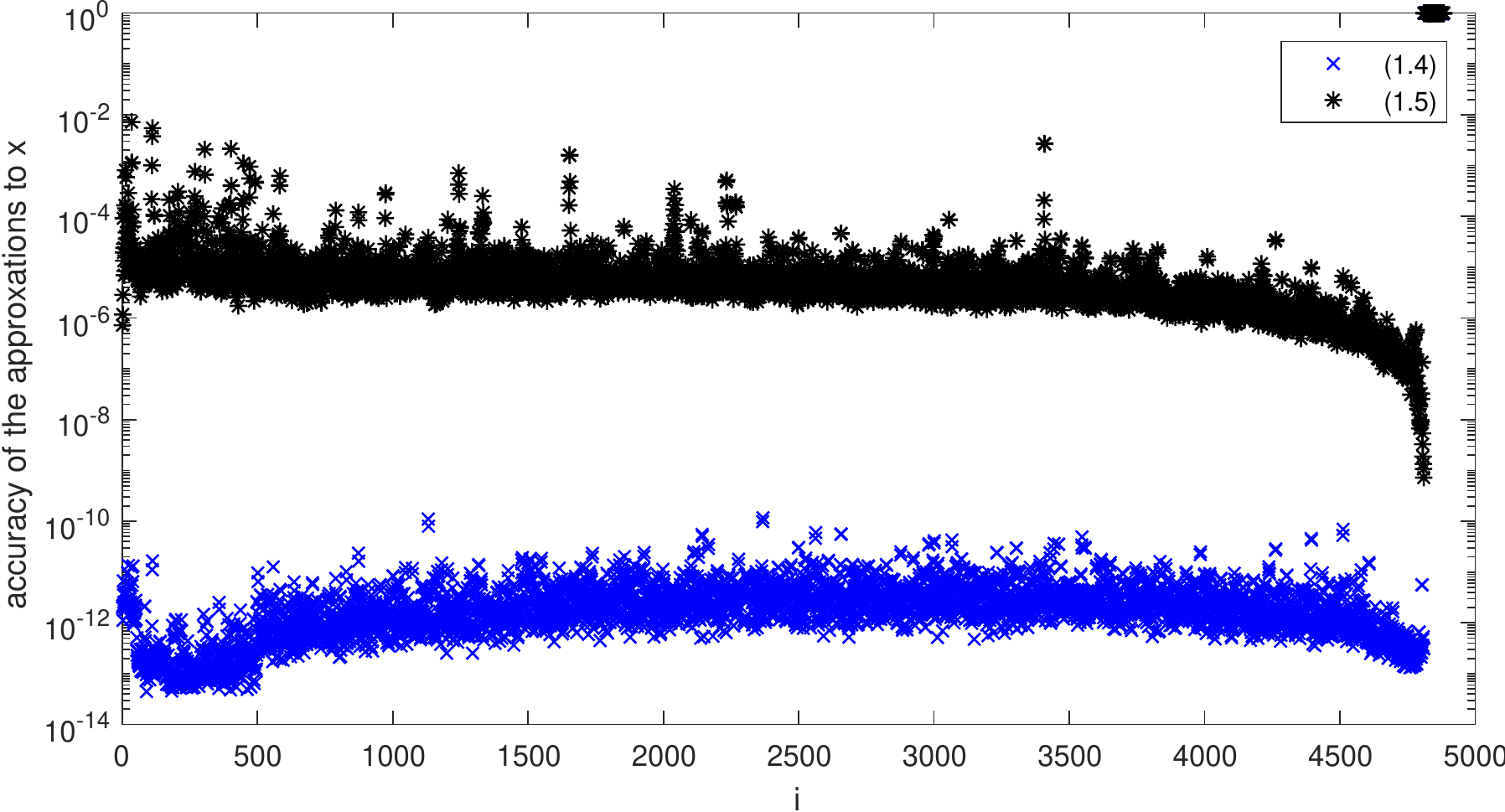}}\\
\subfloat[$\sin\angle(\widehat u,u)$ and
$\sin\angle(\widetilde u,u)$ ]
{\label{fig5c}\includegraphics[width=0.49\textwidth]{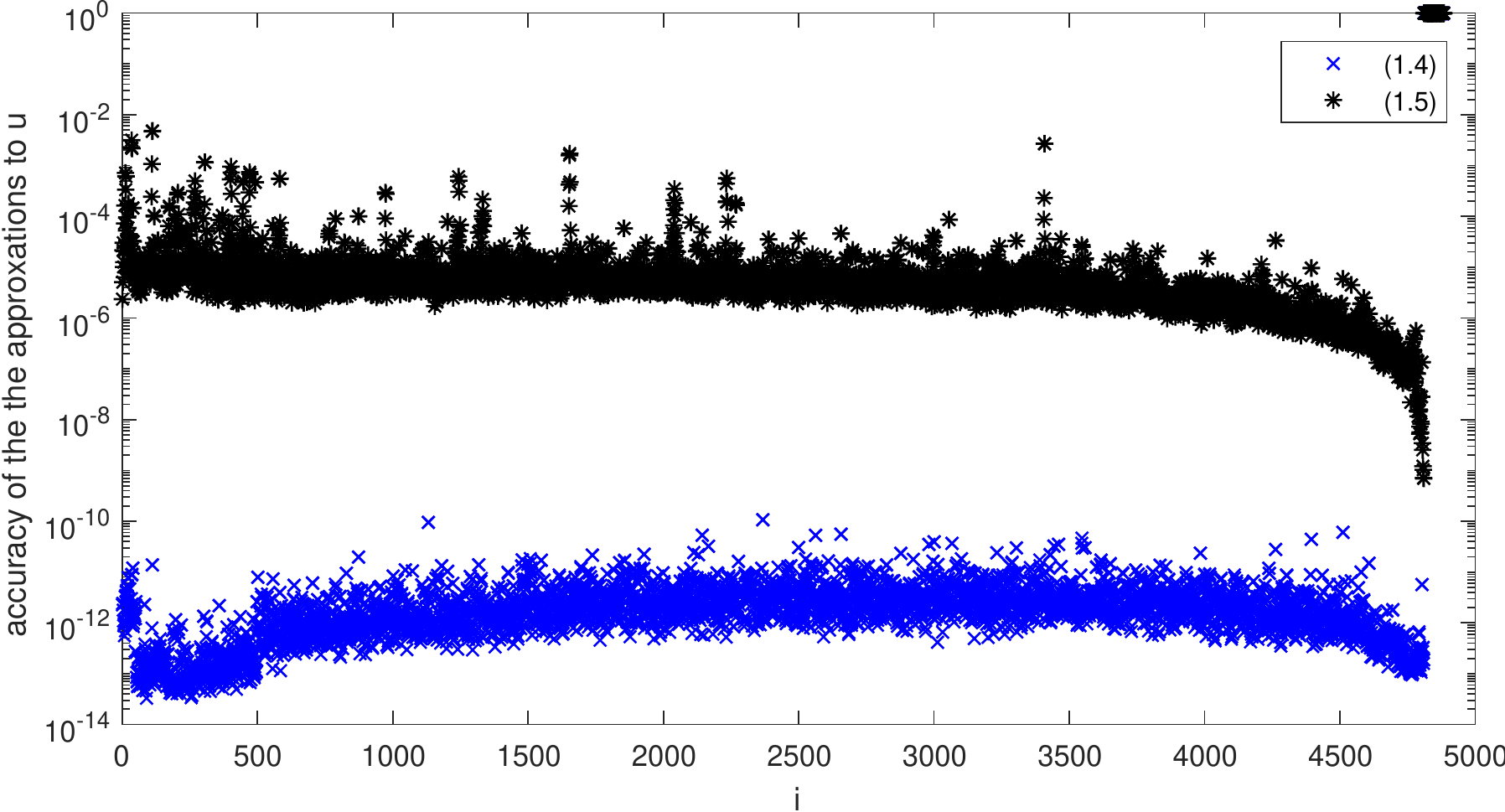}}
\ \ \ \
\subfloat[$\sin\angle(\widehat v,v)$ and
$\sin\angle(\widetilde v,v)$  ]
{\label{fig5d}\includegraphics[width=0.49\textwidth]{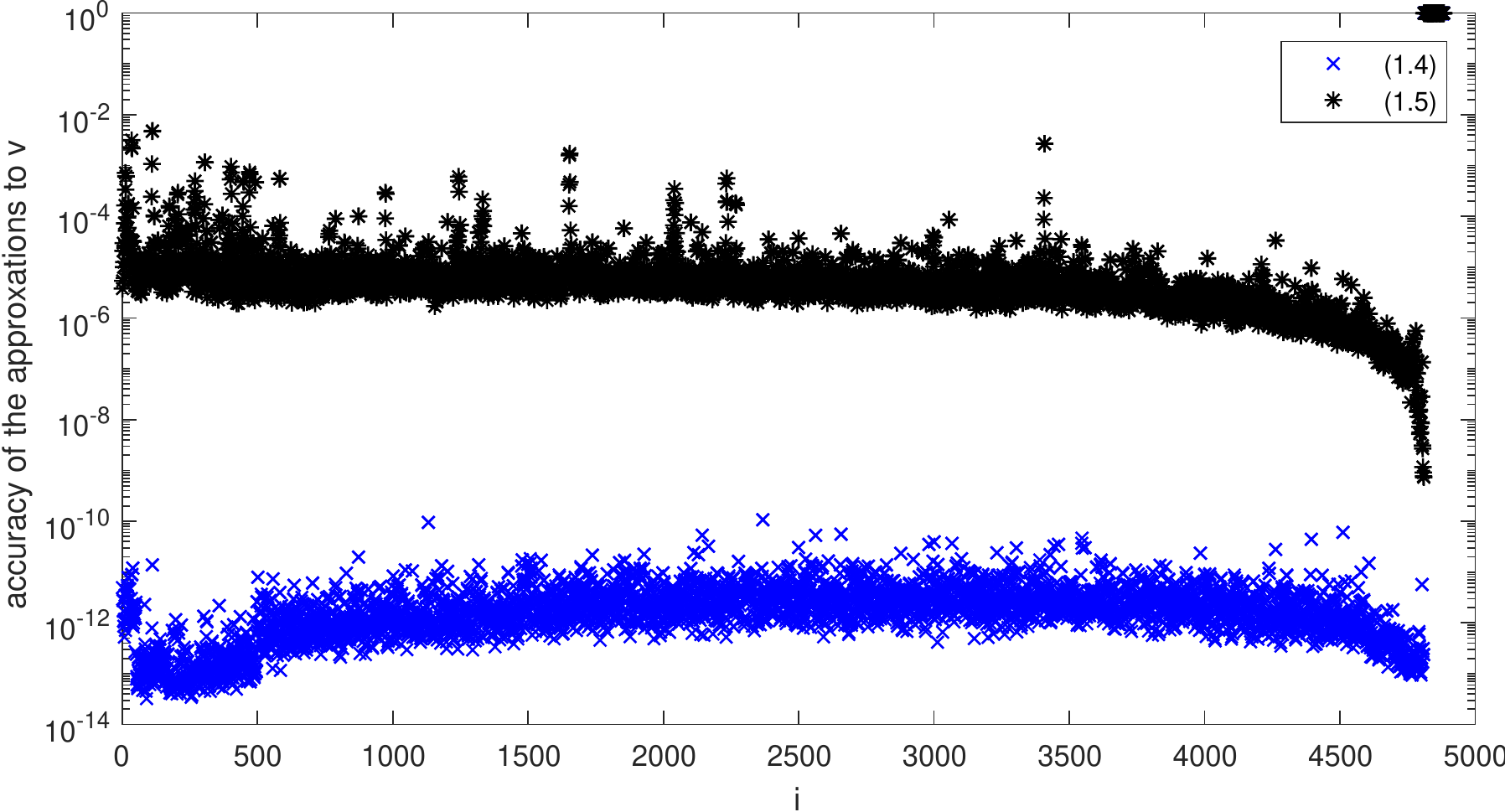}}
\caption{Accuracy of the GSVD components of problem 2f.}\label{fig5}
\end{minipage}
\end{figure}

Table \ref{table3} displays some key data that exhibit the advantages
of \eqref{widehatAB} over \eqref{widetildeBA} when computing the
GSVD of $(A,B)$ more accurately, where $pct$ denotes the
percentages that the computed GSVD components based on
\eqref{widehatAB} are more accurate than those based on \eqref{widetildeBA},
and $acc$ denotes the average orders of magnitude differences
between the accuracy of the computed GSVD components based on
\eqref{widehatAB} and the accuracy of those based on \eqref{widetildeBA},
i.e., $acc$ for the generalized singular values $\sigma$ is defined by
\begin{displaymath}
   acc(\sigma)=\frac{1}{n}\left[\sum_{i=1}^{n}\log\mathcal{X}
   (\widetilde \sigma_i,\sigma_i)-\sum_{i=1}^{n}\log\mathcal{X}
   (\widehat \sigma_i,\sigma_i)\right].
\end{displaymath}
Apparently, the bigger $pct$ and $acc$ are, the more accurate the
GSVD components based on \eqref{widehatAB} are than those based
on \eqref{widetildeBA} on average.
$pct\approx 50\%$ and $acc\approx 0$ indicate that, on average,
there is little
difference and these two formulations based backward stable
eigensolvers compute the GSVD with similar accuracy.

For these six test problems, we have observed very similar
phenomena to the previous experiments.
For problems 2a and 2b where both $A$ and $B$ are equally well
conditioned, \eqref{widehatAB} and \eqref{widetildeBA} are
competitive and there is no obvious winner between them,
though \eqref{widetildeBA} is slightly better than \eqref{widehatAB}.
However, we have seen that, for problems 2c-2f, the matrix
$A$ is increasingly worse conditioned than $B$, the measures
$pct>50\%$ and $acc>0$ increase and become near to one and bigger,
respectively, meaning that more and more GSVD components
are computed more and even much more accurately based on
\eqref{widehatAB}  than on \eqref{widetildeBA}.
Therefore, \eqref{widehatAB} outperforms \eqref{widetildeBA}
for these four problems.

To illustrate the accuracy visually, we depict the results on
problems 2d and 2f in Figures \ref{fig4} and \ref{fig5}, respectively.
For problem 2d, the matrix $B$ is well conditioned and $A$ is ill
conditioned.
We can see from Figure~\ref{fig4} that for the largest  $80\%$
of the GSVD components, \eqref{widehatAB} outperforms
\eqref{widetildeBA} substantially, but for the rest smallest $20\%$
ones, the two formulations are competitive as they yield
comparably accurate approximations.
Particularly,
from Figure \ref{fig4}, we also observe a loss of accuracy of
the approximate generalized singular vectors around the $3500$th GSVD
component.
This occurs because of very small relative gaps
between the corresponding generalized singular values.
For problem 2f where $B$ is well conditioned and $A$ is worse
conditioned,
\eqref{widehatAB} outperforms \eqref{widetildeBA} more
substantially and the accuracy improvements
illustrated by Figure \ref{fig5} are tremendous.
We observe that for almost all (more than 99\%) GSVD
components of $(A,B)$, \eqref{widehatAB} yields much more
accurate approximations than \eqref{widetildeBA} does.
In addition, we see from Figure \ref{fig5} that for the
several smallest GSVD components, using \eqref{widetildeBA}
can compute generalized singular values accurately, but
the corresponding computed generalized singular vectors
have no accuracy at all, while \eqref{widehatAB} works very well.
This is not surprising and is in accordance with our
comments in the near end of Section 2 by noticing that
$\kappa(A)=\mathcal{O}(\epsilon_{\rm mach}^{-1/2})$.

Finally, for all the test problems, we have observed that,
with the suitable formulation chosen and under the chordal
measure, the generalized singular values $\sigma$ are always
computed with full accuracy, which justifies Theorem~\ref{thm:1}
and the analysis followed in Section 2.1.

Summarizing all the experiments, we conclude that (i) both
\eqref{widehatAB} and \eqref{widetildeBA} suit well for
problems where both $A$ and $B$ are well conditioned, (ii)
\eqref{widehatAB} is preferable for problems where $A$
is ill conditioned and $B$ is well conditioned, and (iii)
\eqref{widetildeBA} is preferable for problems where $A$
is well conditioned and $B$ is ill conditioned. Therefore, the
numerical experiments have fully justified our theory.

\section{Conclusions}\label{section:7}

The GSVD of the matrix pair $(A,B)$ can be formulated as two
mathematically equivalent generalized eigenvalue problems of
the matrix pairs defined by
\eqref{widehatAB} and \eqref{widetildeBA},
to which a generalized eigensolver can be applied,
and the GSVD components are recovered from the
computed generalized eigenpairs.
However, in numerical computations,
the two formulations may behave very differently for computing
the GSVD,
and the same generalized eigensolver applied to them may
compute GSVD components with quite different accuracy.
We have made a detailed sensitivity analysis on the generalized
singular values and the generalized singular vectors recovered
from the computed eigenpairs by solving the generalized eigenvalue problems
of the matrix pairs defined by
\eqref{widehatAB} and \eqref{widetildeBA}, respectively.
The results and analysis have shown that
(i) both \eqref{widehatAB} and \eqref{widetildeBA} are
suitable when both $A$ and $B$ are well conditioned;
(ii) \eqref{widehatAB} is preferable when $A$ is ill
conditioned and $B$ is well conditioned;
(iii) \eqref{widetildeBA} suits better when $A$ is well
conditioned and $B$ is ill conditioned.
We have also proposed practical strategies of making a suitable choice
between \eqref{widehatAB} and \eqref{widetildeBA} in
practical computations.

Illuminating numerical experiments have confirmed our theory
and supported our choice strategies on \eqref{widehatAB}
and \eqref{widetildeBA}.


%
%




\end{document}